\documentclass[12pt]{article}
\usepackage[utf8]{inputenc}
\usepackage[margin=1in]{geometry}
\usepackage{graphicx}
\usepackage[table]{xcolor}
\usepackage{subcaption}

\usepackage{amsmath, amsfonts, amssymb, amsthm, amsbsy, mathtools, mathrsfs}

\newtheorem{theorem}{Theorem}[section]
\newtheorem{lemma}[theorem]{Lemma}

\usepackage{algorithm} 
\usepackage{algpseudocode}

\DeclareMathOperator{\tr}{tr}
\DeclareMathOperator*{\argmin}{\arg\!\min}

\usepackage{url, hyperref, lineno}
\usepackage{multirow}
\usepackage{booktabs}

\usepackage{comment}
\usepackage[square,numbers]{natbib}
\begin{document}

\Large
\noindent \textbf{CUR Matrix Approximation through Convex   \newline Optimization for Feature Selection} \newline \newline  
\normalsize
Kathryn Linehan, University of Virginia, Research Computing, Charlottesville, VA, USA \newline  
Radu Balan, University of Maryland, Department of Mathematics, College Park, MD, USA

\begin{abstract}

The singular value decomposition (SVD) is commonly used in applications requiring a low rank matrix approximation. However, the singular vectors cannot be interpreted in terms of the original data.  For applications requiring this type of interpretation, e.g., selection of important data matrix columns or rows, the approximate CUR matrix factorization can be used.  Work on the CUR matrix approximation has generally focused on algorithm development, theoretical guarantees, and applications.  In this work, we present a novel deterministic CUR formulation and algorithm with theoretical convergence guarantees. The algorithm utilizes convex optimization, finds important columns and rows separately, and allows the user to control the number of important columns and rows selected from the original data matrix.  We present numerical results and demonstrate the effectiveness of our CUR algorithm as a feature selection method on gene expression data.  These results are compared to those using the SVD and other CUR algorithms as the feature selection method. Lastly, we present a novel application of CUR as a feature selection method to determine discriminant proteins when clustering protein expression data in a self-organizing map (SOM), and compare the performance of multiple CUR algorithms in this application. \newline

\end{abstract}

\section{Introduction}
\label{sec:intro}

Low rank matrix approximations are common tools in many applications including principal component analysis (PCA), signal denoising, and least squares. While the truncated singular value decomposition (SVD) is the optimal approximation in terms of matrix reconstruction (Eckart-Young theorem), the singular vectors cannot be interpreted in terms of the original data. Mahoney and Drineas \cite{mahoney_drineas_2009} provided an example of this: $[\frac{1}{2}\text{age} - \frac{1}{\sqrt{2}}\text{height} + \frac{1}{2}\text{income}]$ is an eigenvector for a dataset of features about people that ``is not particularly informative or meaningful''.  However, the approximate CUR matrix factorization can be interpreted in terms of the original data, making it an attractive low-rank approximation option, especially for applications that seek important matrix columns or rows. Several of these applications exist \cite{sorenson_embree_2016}, e.g., selecting important genes from gene expression data in order to cluster patients by cancer type \cite{mahoney_drineas_2009}, and more broadly can be considered feature selection applications.  

The approximate CUR factorization of $\mathbf{X} \in \mathbb{R}^{m \times n}$ is generally computed in three steps, but steps (1) and (2) can also be computed simultaneously: (1) select $c \in \mathbb{R}$ columns of $\mathbf{X}$ and let $\mathbf{C} \in \mathbb{R}^{m \times c}$ contain these columns, (2) select $r \in \mathbb{R}$ rows of $\mathbf{X}$ and let $\mathbf{R} \in \mathbb{R}^{r \times n}$ contain these rows, and (3) compute $\mathbf{U} \in \mathbb{R}^{c \times r}$ so that $\mathbf{CUR}$ is a good approximation to $\mathbf{X}$.  The result is a matrix approximation
\begin{equation*}
\underset{m \times n}{\mathbf{X}} \approx \underset{m \times c}{\mathbf{C}} \hspace{0.5em} \underset{c \times r}{\mathbf{U}} \hspace{0.5em} \underset{r \times n}{\mathbf{R}} , 
\end{equation*}
\noindent  where generally $c \ll n$ and $r \ll m$.  Hence, CUR maintains the structure of the data, for example sparsity or nonnegativity, and
 $\mathbf{C}$ and $\mathbf{R}$ can be viewed as containing the most important columns and rows of the original data, respectively.  CUR matrix approximation has been successfully used for feature selection in applications such as document clustering \cite{mahoney_drineas_2009}, gene expression data clustering \cite{mahoney_drineas_2009, sorenson_embree_2016}, image classification \cite{liu_shao_2010}, and sensor selection and channel assignment \cite{esmaeili_etal_2021}.  CUR has also been used for simultaneous feature selection and sample selection for active learning \cite{li_etal_2019}. 

Hamm and Huang \cite{hamm_huang1_2020} provided a history of CUR and mentioned that recent work on the CUR approximation most likely began with developments in the mid-to-late 1990s by Goreinov, Tyrtyshnikov, and Zamarashkin, e.g., \cite{goreinov_etal_1997}.  Since then, several CUR algorithms have been developed; some are randomized, e.g., \cite{drineas_etal_2006, drineas_etal_2008, mahoney_drineas_2009}, and others are deterministic, e.g., \cite{sorenson_embree_2016, stewart_1999}.  Work on CUR includes proving accuracy and/or other theoretical guarantees for algorithms, e.g., \cite{drineas_etal_2006, mahoney_drineas_2009}, and also performance of CUR algorithms in practice without theoretical guarantees, e.g., \cite{mairal_etal_2011}.  In this work, we are particularly interested in deterministic CUR algorithms that can independently select columns of $\mathbf{X}$ without simultaneous selection of its rows due to the fact that (1) several applications exist that seek important matrix columns or rows and not a full matrix factorization \cite{sorenson_embree_2016}, and (2) for a practical application, a randomized CUR will likely produce a different set of important columns and/or rows in each run of the algorithm, which may not be desirable to the scientist \cite{bien_xu_mahoney_2010, mairal_etal_2011}.  

One deterministic approach to computing a CUR approximation is to select columns and rows of $\mathbf{X}$ for inclusion in $\mathbf{C}$ and $\mathbf{R}$ using convex optimization with regularization, e.g., \cite{bien_xu_mahoney_2010, mairal_etal_2011}.  In this work, we propose a novel CUR algorithm utilizing convex optimization with contributions in the formulation, implementation, and application of CUR. The main contributions of the work include (1) a novel convex optimization formulation for CUR, (2) an algorithm utilizing convex optimization that solves for $\mathbf{C}$ and $\mathbf{R}$ separately and allows the user to select $c$ and $r$, and (3) an implementation utilizing the ``surrogate functional'' technique of Daubechies, Defrise, and De Mol \cite{daubechies_etal_2004}, which we adapt for use with a new penalty function.  We also note that our CUR algorithm and implementation can accommodate a variety of penalty functions, allowing the user a flexible framework.  We provide numerical results that compare our CUR algorithm with the SVD and other deterministic CUR algorithms that select $\mathbf{C}$ and $\mathbf{R}$ separately and allow the user to select $c$ and $r$.  Specifically, we show that our CUR algorithm performs very well as a feature selection method in an extension of an experiment by Sorenson and Embree \cite{sorenson_embree_2016} on gene expression data, in which important genes are selected to cluster patients into two classes - those with and without a lung tumor.

Another main contribution of the work is a novel application of CUR for feature selection. We adapt the experiments of Higuera et al. \cite{higuera_et_al_2015} in which Self-Organizing Maps (SOMs)\footnote{also known as Kohonen Maps.} and the Wilcoxon rank-sum test were used to determine proteins that critically affect learning in wild type and trisomic (Down syndrome) mice.  Specifically, we use CUR as the feature selection method instead of the Wilcoxon rank-sum test.  We show that CUR can be used effectively in this application and compare the performance of our CUR algorithm to that of other deterministic CUR algorithms that select $\mathbf{C}$ and $\mathbf{R}$ separately and allow the user to select $c$ and $r$.  This is not only a novel application of CUR, but to the best of our knowledge, also the first use of CUR on protein expression data. 

The remainder of this paper is organized as follows: we present related work in Section \ref{sec:cur_rel_work}, our novel CUR algorithm utilizing convex optimization in Section \ref{sec:cur_alg}, the theoretical foundations of the algorithm in Section \ref{sec:cur_thms}, numerical experiments in Section \ref{sec:exp}, a novel application of CUR as a feature selection method in protein expression discriminant analysis in Section \ref{sec:pro_exp}, and a conclusion in Section \ref{sec:cur_conc}.  Throughout this work we use MATLAB notation to denote rows and columns of matrices, e.g., row $i$ of $\mathbf{X}$ is denoted $\mathbf{X}(i,:)$ and column $j$ of $\mathbf{X}$ is denoted $\mathbf{X}(:,j)$.  In addition, the set $\{1, 2, ..., n \}$ is denoted $[n]$. 

\section{Related Work}
\label{sec:cur_rel_work}

As mentioned in Section \ref{sec:intro}, work on the CUR matrix approximation has generally focused on algorithm development, theoretical guarantees, and applications. In this section, we focus on related work in three areas: (1) deterministic CUR algorithms that solve for $\mathbf{C}$ and $\mathbf{R}$ separately and allow the user to select $c$ and $r$, (2) CUR algorithms that use convex optimization with regularization to select columns and rows of the data matrix for inclusion in $\mathbf{C}$ and $\mathbf{R}$, and (3) CUR feature selection applications.  Since we mentioned a number of feature selection applications in the introduction, here we will provide more details.  For the interested reader, Dong and Martinsson \cite{dong_martinsson_2023} provide a survey of CUR algorithms (including those that do not fit the criterion for inclusion in this section). 

Deterministic CUR algorithms that solve for $\mathbf{C}$ and $\mathbf{R}$ separately and allow the user to select $c$ and $r$ include a leverage score approach \cite{mahoney_drineas_2009}, a discrete empirical interpolation method (DEIM) approach \cite{sorenson_embree_2016}, and a pivoted QR approach \cite{stewart_1999}.  The leverage score approach by Mahoney and Drineas \cite{mahoney_drineas_2009} is often compared to in the CUR literature. This approximation is randomized and columns and rows are sampled based on their ``normalized statistical leverage scores'', which capture information on how much a column or row contributes to the optimal low-rank approximation to the data matrix, the rank-$k$ SVD, where $k$ is a rank parameter chosen by the user.  However, a deterministic variant of this algorithm is to select the columns and rows with the largest leverage scores for inclusion in $\mathbf{C}$ and $\mathbf{R}$, respectively.  In the DEIM approach by Sorenson and Embree \cite{sorenson_embree_2016}, columns are chosen for inclusion in $\mathbf{C}$ and rows are chosen for inclusion in $\mathbf{R}$ using the discrete empirical interpolation method (DEIM) on the top-$k$ right and left singular vectors of the data matrix, respectively, where $k=c=r$.  In the pivoted QR approach by Stewart \cite{stewart_1999}, columns are selected for inclusion in $\mathbf{C}$ and rows are selected for inclusion in $\mathbf{R}$ using a pivoted QR factorization of $\mathbf{X}$ and $\mathbf{X}^T$, respectively. Sorenson and Embree \cite{sorenson_embree_2016} present a slight adaptation of this approach in which rows are selected for inclusion in $\mathbf{R}$ using a pivoted QR factorization of $\mathbf{C}^T$.  In each of these four CUR algorithms, $\mathbf{U}$ is computed as $\mathbf{U = C^+XR^+}$, i.e., the minimizer of $\|\mathbf{X} - \mathbf{CUR} \|_F$ \cite{stewart_1999}. 

Since the CUR algorithm that we present in this work uses convex optimization with regularization to select columns and rows of the data matrix for inclusion in $\mathbf{C}$ and $\mathbf{R}$, we also note related work in this area. In \cite{bien_xu_mahoney_2010} Bien, Xu, and Mahoney related CUR to sparse PCA and used the following convex minimization problem to find $\mathbf{C}$:
\begin{equation}
\label{eqn:c_conopt}
\mathbf{B}^* = \argmin_{\mathbf{B} \in \mathbb{R}^{n \times n}} \|\mathbf{X} - \mathbf{XB}\|_F + \lambda \sum_{i=1}^n \| \mathbf{B}(i,:) \|_2,
\end{equation}
where $\lambda>0$ is a regularization parameter. The indices of nonzero rows of $\mathbf{B}^*$ are the indices of columns to choose from $\mathbf{X}$ for inclusion in $\mathbf{C}$. $\mathbf{R}$ can be found using a similar optimization problem; however, the computation of $\mathbf{U}$ is not discussed. Mairal et al. \cite{mairal_etal_2011} formulated CUR as a convex optimization problem that selects columns and rows of $\mathbf{X}$ at the same time:
\begin{equation}
\label{eqn:cur_sim}    
\mathbf{W^*} = \argmin_{\mathbf{W} \in \mathbb{R}^{n \times m}} ||\mathbf{X-XWX}||_F^2 + \lambda_{\text{row}} \sum_{i=1}^m \| \mathbf{W}(:,i) \|_{\infty} + \lambda_{\text{col}} \sum_{j=1}^n \| \mathbf{W}(j,:) \|_{\infty},
\end{equation}
where $\lambda_{\text{row}}, \lambda_{\text{col}}>0$ are regularization parameters.  Similar to \cite{bien_xu_mahoney_2010}, the nonzero row indices of $\mathbf{W}^*$ are the indices of columns to select from $\mathbf{X}$, and the nonzero column indices of $\mathbf{W}^*$ are the indices of rows to select from $\mathbf{X}$. $\mathbf{U}$ is computed as $\mathbf{U = C^+XR^+}$.  In both \cite{bien_xu_mahoney_2010, mairal_etal_2011}, the convex optimization CUR algorithm achieves similar matrix reconstruction accuracy to that of the leverage score based CUR \cite{mahoney_drineas_2009} in numerical experiments.  In addition, Ida et al. \cite{ida_etal_2020} presented a method to speed up the coordinate descent algorithm for solving Equation \ref{eqn:c_conopt} as presented in \cite{bien_xu_mahoney_2010} and claimed it can be extended to solve Equation \ref{eqn:cur_sim} as well.  

In \cite{peng_etal_2018}, Peng et al. used an optimization problem with regularization terms that simultaneously performed a CUR approximation of a network node attribute matrix (to choose representative nodes and attributes at the same time) and residual analysis in order to detect anomalies on attributed networks.  The part of the optimization formulation related to CUR is similar to Equation \ref{eqn:cur_sim}, but uses the $\ell_2$ norm rather than the $\ell_{\infty}$ norm in the regularization terms.  This optimization problem was solved using alternating convex optimizations, and parameters were chosen using a grid search in experimental results.  In each of the convex optimization CUR approaches mentioned above \cite{bien_xu_mahoney_2010, mairal_etal_2011, ida_etal_2020, peng_etal_2018} there is not a built-in algorithmic control for selecting $c$ columns and $r$ rows of the data matrix. 

Li et al. \cite{li_etal_2019} used a convex optimization CUR to simultaneously select features and representative data samples in order to perform feature selection and active learning at the same time.  The optimization problem used is similar to Equation \ref{eqn:cur_sim} except that the regularization terms use the $\ell_2$ norm rather than the $\ell_{\infty}$ norm, and there is an additional regularization term that provides ``local linear reconstruction''.  Parameters were grid searched in experimental results and after the optimization problem is solved, the indices of the $c$ rows of $\mathbf{W}^*$ with the largest $\ell_{2}$ norms are the indices of columns selected from the data matrix for inclusion in $\mathbf{C}$.  The indices of $r$ rows to include in $\mathbf{R}$ are found similarly.  

We provided examples of applications that use CUR for feature selection in Section \ref{sec:intro}, and here provide more details for those examples.  Mahoney and Drineas \cite{mahoney_drineas_2009} applied CUR to a term-document matrix and used the results to cluster the documents into two topics, with interpretation provided by selection of the five most important terms by CUR.  The clustering provided by CUR outperformed that provided by the truncated SVD.  They also similarly applied CUR to gene expression data to cluster patients by cancer type.  Clustering peformance was equivalent to that of using the truncated SVD, but CUR provided insight into which genes are most important to the clustering and of the twelve selected, some are known to be medically associated with cancer.  Sorenson and Embree \cite{sorenson_embree_2016} also used CUR on gene expression data to discover genes that cluster patients into those with and without a tumor.  They compared results using their DEIM CUR and the deterministic leverage score CUR of Mahoney and Drineas \cite{mahoney_drineas_2009}.  While the DEIM CUR reconstructed the original data matrix better than the deterministic leverage score CUR, the genes selected by the leverage score CUR performed much better in separating patients with and without a tumor.  

Liu and Shao \cite{liu_shao_2010} leveraged CUR for feature selection to improve image classification accuracy.  CUR performed the best of the dimensionality reduction methods used, which included PCA. Esmaeili et al. \cite{esmaeili_etal_2021} used CUR for cognitive radio sensor selection and channel assignment.  Specifically, sensors were chosen using the selected columns from $\mathbf{C}$, channels were selected using the highest magnitude elements of $\mathbf{U}$ for each chosen sensor, and the resulting samples were interpolated to create the spectrum map.  They tested various CUR algorithms including the leverage score CUR \cite{mahoney_drineas_2009} and showed that CUR is more effective than random uniform sampling (the prior method) in recreating the spectrum map.  As mentioned above, Li et al. \cite{li_etal_2019} used CUR for simultaneous feature selection and sample selection for active learning to classify synthetic data, gene expression data, molecular data, image and video data, and human activity recognition data.  They demonstrated that their convex optimization CUR almost always outperformed other feature selection and active learning methods, including the randomized leverage score CUR \cite{mahoney_drineas_2009}, in terms of classifier accuracy.

\section{CUR Algorithm}
\label{sec:cur_alg}

Let $\mathbf{X} \in \mathbb{R}^{m \times n}$ be the matrix we wish to approximate as $\mathbf{X} = \mathbf{CUR}$. Our formulation of CUR using convex optimization builds upon ideas from  \cite{bien_xu_mahoney_2010}, \cite{mahoney_drineas_2009}, and \cite{mairal_etal_2011}. To select a subset of columns from $\mathbf{X}$ to form the matrix $\mathbf{C}$, we solve 
\begin{equation}
\label{eqn:C}
\mathbf{W^*} = \argmin_{\mathbf{W} \in \mathbb{R}^{n \times m}} ||\mathbf{X-XWX}||_F^2 + \lambda_C \sum_{i=1}^n \| \mathbf{W}(i,:) \|_{\infty},
\end{equation}
for a given $\lambda_C \in \mathbb{R} \geq 0$.  Then $\mathbf{C} = \mathbf{X}(:, I_C)$, where $I_C$ is the set of indices of nonzero rows in $\mathbf{W^*}$.  Hence $\lambda_C$ controls how many columns are selected from $\mathbf{X}$. i.e., the larger the value of $\lambda_C$, the more rows of $\mathbf{W}$ are forced to $\mathbf{0}$, and the fewer columns of $\mathbf{X}$ are selected.

After $\mathbf{C}$ has been calculated, we select a set of rows from $\mathbf{X}$ to form the matrix $\mathbf{R}$ by solving   
\begin{equation}
\label{eqn:R}
\mathbf{W^*} = \argmin_{\mathbf{W} \in \mathbb{R}^{c \times m}} ||\mathbf{X-CWX}||_F^2 + \lambda_R \sum_{j=1}^m \| \mathbf{W}(:,j) \|_{\infty},
\end{equation}
for a given $\lambda_R \in \mathbb{R} \geq 0$.  Then $\mathbf{R} = \mathbf{X}(I_R,:)$, where $I_R$ is the set of indices of nonzero columns in $\mathbf{W^*}$, and $\lambda_R$ controls the number of rows of selected.  

Input to our algorithm includes $c$ and $r$, the number of columns and rows to be selected from $\mathbf{X}$ for $\mathbf{C}$ and $\mathbf{R}$, respectively.  For a given $\lambda_C$, it is unknown in advance how many columns will be selected from $\mathbf{X}$ by the solution of Equation \ref{eqn:C}.  Hence, we utilize bisection on $\lambda_C$ with multiple iterates of the column selection procedure to find a selection of exactly $c$ columns.  We use a similar process with $\lambda_R$ and the row selection procedure to find a selection of exactly $r$ rows. To complete the algorithm, $\mathbf{U}$ is computed as $\mathbf{U = C^+XR^+}$, where $\mathbf{X}^+$ denotes the Moore-Penrose generalized inverse or pseudoinverse of $\mathbf{X}$. The pseudocode for our CUR approximation is given in Algorithm \ref{alg:cur_co}.

The initial minimum value for $\lambda_C$ in the bisection method is $0$, which corresponds to potentially all columns of $\mathbf{X}$ being selected for the matrix $\mathbf{C}$.  The initial maximum value for $\lambda_C$ in the bisection method is the smallest value of $\lambda_C$ that forces zero columns of $\mathbf{X}$ to be selected, i.e., the solution to Equation \ref{eqn:C} to be $\mathbf{W} = \mathbf{0}$.  We call this the critical value of $\lambda_C$ and denote it $\lambda_C^*$.  The range of values for $\lambda_R$ in the bisection method with Equation \ref{eqn:R} is set similarly, with the critical value of $\lambda_R$ denoted $\lambda_R^*$.  To prove the exact values of $\lambda_C^*$ and $\lambda_R^*$, we first provide a helpful lemma and note that Equations \ref{eqn:C} and \ref{eqn:R} can be reshaped as
\begin{equation}
\label{eqn:C_reshape}
\mathbf{w^*} = \min_{\mathbf{w} \in \mathbb{R}^{mn}} ||\mathbf{(X^T \otimes X) w}-\mathbf{b}||^2_2 + \lambda_C \sum_{i=1}^n \max_{1 \leq j \leq m} |\mathbf{w}_{i+(j-1)n}|,
\end{equation}
and 
\begin{equation}
\label{eqn:R_reshape}
\mathbf{w^*} = \min_{\mathbf{w} \in \mathbb{R}^{mc}} ||\mathbf{(\mathbf{X}^T \otimes \mathbf{C})w}-\mathbf{b}||^2_2 + \lambda_R \sum_{j=1}^m \max_{1 \leq i \leq c} |\mathbf{w}_{i+(j-1)n}|,
\end{equation}
where $\otimes$ denotes the Kronecker product, $\mathbf{b} = \text{vec}(\mathbf{X}) \in \mathbb{R}^{mn}$, i.e., a column-stacked version of $\mathbf{X}$, $\mathbf{w} = \text{vec}(\mathbf{W}) \in \mathbb{R}^{mn}$ (Equation \ref{eqn:C_reshape}) or $\mathbf{w} = \text{vec}(\mathbf{W}) \in \mathbb{R}^{mc}$ (Equation \ref{eqn:R_reshape}), and $\mathbf{w}^*$ is defined similarly to $\mathbf{w}$. 

\begin{lemma}
\label{lemma:rx}
Let $\mathbf{v} \in \mathbb{R}^{m}$ be fixed, and $\lambda \in \mathbb{R}$.  If $\langle \mathbf{v},\mathbf{x} \rangle \leq \frac{\lambda}{2}||\mathbf{x}||_\infty$ $\forall \mathbf{x} \in \mathbb{R}^{m}$, then $\lambda \geq 2 ||\mathbf{v}||_1$.
\end{lemma}

\begin{proof}
\[
\langle \mathbf{v},\mathbf{x} \rangle = \sum_{k=1}^m \mathbf{v}_k \mathbf{x}_k \leq \left( \sum_{k=1}^m |\mathbf{v}_k| \right) \max_{1 \leq k \leq m} |\mathbf{x}_k| = ||\mathbf{v}||_1 ||\mathbf{x}||_\infty.
\]
Hence for $\mathbf{x}_k = \text{sign}(\mathbf{v}_k)$, $\langle \mathbf{v},\mathbf{x} \rangle = ||\mathbf{v}||_1 ||\mathbf{x}||_\infty$ and $\lambda = 2||\mathbf{v}||_1$.  Since the smallest value of $\lambda$ which holds $\forall \mathbf{x} \in \mathbb{R}^{m}$ will occur when $\langle \mathbf{v},\mathbf{x} \rangle = \frac{\lambda}{2}||\mathbf{x}||_\infty$, it is the case that $\forall \mathbf{x} \in \mathbb{R}^{m}$, $\lambda \geq 2 ||\mathbf{v}||_1$.
\end{proof}

\begin{theorem}
Let $\mathbf{M} = \text{\emph{reshape}}(\mathbf{(X^T \otimes X)^Tb}, n, m)$, i.e., $\mathbf{(X^T \otimes X)^Tb}$ reshaped from $\mathbb{R}^{mn}$ to $\mathbb{R}^{n \times m}$ so that $\mathbf{(X^T \otimes X)^Tb} = \text{vec}(\mathbf{M})$.  Then 
\[
\lambda_C^* = 2 \max_{1 \leq i \leq n} \| \mathbf{M}(i,:) \|_1 = 2||\mathbf{M}||_\infty.
\]
Similarly, let $\mathbf{N} = \text{\emph{reshape}}(\mathbf{(X^T \otimes C)^Tb}, c, m)$.  Then 
\[
\lambda_R^* = 2 \max_{1 \leq j \leq m} \| \mathbf{N}(:,j) \|_1 = 2||\mathbf{N}||_1. 
\] 
\end{theorem}

\begin{proof}
Let $\mathbf{A} = \mathbf{(X^T \otimes X)}$ and the objective function of Equation \ref{eqn:C_reshape} be $J(\mathbf{w})$:
\[
J(\mathbf{w}) = ||\mathbf{Aw}-\mathbf{b}||^2_2 + \lambda_C \sum_{i=1}^n \max_{1 \leq j \leq m} |\mathbf{w}_{i+(j-1)n}|.
\]
We want to find the smallest $\lambda_C^* > 0$ such that $\forall \lambda \geq \lambda_C^*$, $\argmin_{\mathbf{w} \in \mathbb{R}^{mn}} J(\mathbf{w}) = \mathbf{0}$.  We have

\begin{equation*}
\begin{split}
J(\mathbf{w}) & = ||\mathbf{Aw}||^2_2 - 2\mathbf{w^TA^Tb} + ||\mathbf{b}||^2_2 + \lambda_C \sum_{i=1}^n \max_{1 \leq j \leq m} |\mathbf{w}_{i+(j-1)n}| \\
& = ||\mathbf{Aw}||^2_2 + ||\mathbf{b}||^2_2  + \lambda_C \sum_{i=1}^n \max_{1 \leq j \leq m} |\mathbf{w}_{i+(j-1)n}| - 2 \sum_{k=1}^{mn} (\mathbf{A^Tb})_k \mathbf{w}_k \\
& = ||\mathbf{Aw}||^2_2 + ||\mathbf{b}||^2_2  + \lambda_C \sum_{i=1}^n \max_{1 \leq j \leq m} |\mathbf{w}_{i+(j-1)n}| - 2 \sum_{i=1}^{n} \sum_{j=1}^{m} \mathbf{M}_{ij} \mathbf{W}_{ij}, \\
\end{split}
\end{equation*}
where $\mathbf{W} = \text{reshape}(\mathbf{w}, n, m)$.  If $\forall i$, 
\begin{equation}
\label{eqn:crit_lambda_lemma}
\frac{\lambda_C}{2} \max_{1 \leq j \leq m} |\mathbf{w}_{i+(j-1)n}| - \sum_{j=1}^{m} \mathbf{M}_{ij} \mathbf{W}_{ij} \geq 0,
\end{equation}
then 
\[
\sum_{i=1}^n \left[ \frac{\lambda_C}{2} \max_{1 \leq j \leq m} |\mathbf{w}_{i+(j-1)n}| - \sum_{j=1}^{m} \mathbf{M}_{ij} \mathbf{W}_{ij} \right] \geq 0, 
\]
and 
\[
\lambda_C \sum_{i=1}^n \max_{1 \leq j \leq m} |\mathbf{w}_{i+(j-1)n}| - 2 \sum_{i=1}^{n} \sum_{j=1}^{m} \mathbf{M}_{ij} \mathbf{W}_{ij} \geq 0.
\]
Hence, $J(\mathbf{w}) \geq J(\mathbf{0}) = ||\mathbf{b}||^2_2$ for all $\mathbf{w}$, thus $\argmin_{\mathbf{w} \in \mathbb{R}^{mn}} J(\mathbf{w}) = \mathbf{0}$.  By Lemma \ref{lemma:rx}, assuming Equation \ref{eqn:crit_lambda_lemma} is true, $\forall i$, $\lambda_C \geq 2||\mathbf{M}(i,:)||_1$.  Thus, $\lambda_C^* = 2 \max_{1 \leq i \leq n} \| \mathbf{M}(i,:) \|_1 = 2||\mathbf{M}||_\infty$.   

A similar proof can be used to show the result for $\lambda_R^*$, letting $\mathbf{A} = \mathbf{(X^T \otimes C)}$ and $J(\mathbf{w})$ be the objective function of Equation \ref{eqn:R_reshape}:
\[
J(\mathbf{w}) = ||\mathbf{Aw}-\mathbf{b}||^2_2 + \lambda_R \sum_{j=1}^m \max_{1 \leq i \leq c} |\mathbf{w}_{i+(j-1)n}|.
\]
\end{proof}

\begin{algorithm}
\caption{CUR through Convex Optimization}
\begin{algorithmic}[1]
\Require $\mathbf{X} \in \mathbb{R}^{m \times n}$, $c > 0$, $r > 0$
\Ensure $\mathbf{C} \in \mathbb{R}^{m \times c}$, $\mathbf{U} \in \mathbb{R}^{c \times r}$, $\mathbf{R} \in \mathbb{R}^{r \times n}$ such that $\mathbf{X \approx CUR}$
\State $\lambda_C^* = \| \mathbf{X}^T \mathbf{X} \mathbf{X}^T  \|_{\infty}$
\State $n_c = 0$, $\lambda_{\text{min}} = 0$, $\lambda_{\text{max}} = \lambda_C^*$
\While {$n_c \neq c$}
    \State $\lambda_{C} = (\lambda_{\text{max}} + \lambda_{\text{min}})/2$
    \State solve $\mathbf{W^*} = \argmin_{\mathbf{W} \in \mathbb{R}^{n \times m}} ||\mathbf{X-XWX}||_F^2 + \lambda_C \sum_{i=1}^n  \| \mathbf{W}(i,:) \|_{\infty}$
    \State let $I_C$ be the set of indices of nonzero rows of $\mathbf{W^*}$
    \State $n_c = |I_C|$
    \If {$c > n_c$}
        \State $\lambda_{\text{min}} = \lambda_C$
    \ElsIf {$c < n_c$}
        \State $\lambda_{\text{max}} = \lambda_C$
    \EndIf
\EndWhile
\State $\mathbf{C} = \mathbf{X}(:,I_C)$
\State $\lambda_R^* = \| \mathbf{C}^T \mathbf{X} \mathbf{X}^T  \|_{1}$
\State $n_r = 0$, $\lambda_{\text{min}} = 0$, $\lambda_{\text{max}} = \lambda_R^*$
\While {$n_r \neq r$}
    \State $\lambda_{R} = (\lambda_{\text{max}} + \lambda_{\text{min}})/2$
    \State solve $\mathbf{W^*} = \argmin_{\mathbf{W} \in \mathbb{R}^{c \times m}} ||\mathbf{X-CWX}||_F^2 + \lambda_R \sum_{j=1}^m \| \mathbf{W}(:,j) \|_{\infty}$
    \State let $I_R$ be the set of indices of nonzero columns of $\mathbf{W^*}$
    \State $n_r = |I_R|$
    \If {$r > n_r$}
        \State $\lambda_{\text{min}} = \lambda_R$
    \ElsIf {$r < n_r$}
        \State $\lambda_{\text{max}} = \lambda_R$
    \EndIf
\EndWhile
\State $\mathbf{R} = \mathbf{X}(I_R,:)$
\State $\mathbf{U = C^+XR^+}$
\State \Return $\mathbf{C,U,R}$
\end{algorithmic}
\label{alg:cur_co}
\end{algorithm}

\subsection{Implementation for Minimization Problems}
\label{sec:C_min}

For $\mathbf{X}$ of small\footnote{Small is relative to the user's computer memory size.} dimensions, we can solve the minimization problems on lines 5 and 19 of Algorithm \ref{alg:cur_co} using the reshaped versions (Equations \ref{eqn:C_reshape} and \ref{eqn:R_reshape}) and a convex programming solver such as the CVX package in MATLAB \cite{cvx_web, cvx_book}. However using a solver for $\mathbf{X}$ of larger dimensions becomes infeasible due to the use of the Kronecker product in the reshaped problems. For example, to store the genetics dataset referenced in Section \ref{sec:exp}, a dense double array $\mathbf{X} \in \mathbb{R}^{\text{107} \times \text{22,283}}$, 19.07 MB are used; to store $\mathbf{X}^T \otimes \mathbf{X} \in \mathbb{R}^{2,384,281 \times 2,384,281}$, 45.48 TB are used.  

To accommodate large-scale problems we solve these minimizations in Algorithm \ref{alg:cur_co} using an extension of an iterative method by Daubechies, Defrise, and De Mol \cite{daubechies_etal_2004} that utilizes a ``surrogate functional'' to solve regularized least squares minimization problems in which weighted $\ell_p$-norm penalty functions are used, for $1 \leq p \leq 2$. This technique decouples a large minimization problem into smaller, easy-to-solve problems.  We extend the results of \cite{daubechies_etal_2004} to apply to our penalty functions, e.g., $\sum_{i=1}^n \| \mathbf{W}(i,:) \|_{\infty}$.  We demonstrate the method for the line 5 minimization in Algorithm \ref{alg:cur_co}.  The line 19 minimization is handled similarly.

Let the objective function of Equation \ref{eqn:C} be denoted by   
\begin{equation*}
\begin{split}
J(\mathbf{W}) = ||\mathbf{X-XWX}||_F^2 + \lambda_C \sum_{i=1}^n \| \mathbf{W}(i,:) \|_{\infty}, \\
\end{split}
\end{equation*}
and the corresponding surrogate functional by   
\begin{equation*}
\label{eqn:surr}
\widehat{J}(\mathbf{W}, \mathbf{Z}) = ||\mathbf{X-XWX}||_F^2 + \lambda_C \sum_{i=1}^n \| \mathbf{W}(i,:) \|_{\infty} + \mu ||\mathbf{W}-\mathbf{Z}||_F^2 - ||\mathbf{XWX}-\mathbf{XZX}||_F^2, 
\end{equation*}
where $\mathbf{Z} \in \mathbb{R}^{n \times m}$ and $\mu > 0$. For any $\mathbf{Z} \in \mathbb{R}^{n \times m}$ and $\mu > 0$ we have
\begin{equation*}
\begin{split}
\widehat{J}(\mathbf{W}, \mathbf{Z}) 
& = \mu ||\mathbf{W}||_F^2 -2\tr\{\mathbf{W}(\mu \mathbf{Z^T} + \mathbf{XX^TX} - \mathbf{XX^T Z^T X^TX})\} + \lambda_C \sum_{i=1}^n \| \mathbf{W}(i,:) \|_{\infty} \\
& \quad + ||\mathbf{X}||_F^2 +  \mu ||\mathbf{Z}||_F^2 - ||\mathbf{XZX}||_F^2 \\
& = \sum_{i=1}^n \left[\mu ||\mathbf{W}(i,:)||_2^2 - 2\langle \mathbf{W}(i,:), \mathbf{L^T}(i,:) \rangle + \lambda_C ||\mathbf{W}(i,:)||_{\infty} \right] \\
& \quad + ||\mathbf{X}||_F^2 +  \mu ||\mathbf{Z}||_F^2 - ||\mathbf{XZX}||_F^2, \\
\end{split}
\end{equation*}
where $\mathbf{L} = \mu \mathbf{Z^T} + \mathbf{XX^TX} - \mathbf{XX^T Z^T X^TX}$.  Hence,
\begin{equation*}
\begin{split}
\argmin_{\mathbf{W} \in \mathbb{R}^{n \times m}} \widehat{J}(\mathbf{W}, \mathbf{Z}) & = 
\argmin_{\mathbf{W} \in \mathbb{R}^{n \times m}} \mu \sum_{i=1}^n \left[||\mathbf{W}(i,:) - \frac{1}{\mu}\mathbf{L^T}(i,:)||^2_2 - ||\frac{1}{\mu}\mathbf{L^T}(i,:)||^2_2 + \frac{\lambda_C}{\mu}||\mathbf{W}(i,:)||_{\infty} \right] \\
& = \argmin_{\mathbf{W} \in \mathbb{R}^{n \times m}} 2\mu \sum_{i=1}^n \left[\frac{1}{2}||\mathbf{W}(i,:) - \frac{1}{\mu}\mathbf{L^T}(i,:)||^2_2 + \frac{\lambda_C}{2\mu}||\mathbf{W}(i,:)||_{\infty} \right]. 
\end{split}
\end{equation*}
Thus, we can easily minimize $\widehat{J}$ over $\mathbf{W}$ by computing the proximal operator of the $\ell_\infty$ norm, 
\begin{equation}
\label{eqn:cur_proxop}
\textbf{prox}_{\alpha ||\cdot||_\infty}(\mathbf{x}) = \argmin_{\mathbf{y} \in \mathbb{R}^m} \left(\frac{1}{2}||\mathbf{y-x}||_2^2 + \alpha ||\mathbf{y}||_\infty \right)
\end{equation}
for $\mathbf{x} \in \mathbb{R}^m$ and $\alpha \geq 0$, for each row of $\mathbf{W}$.  To find a minimizer of $J$, we utilize the minimization of $\widehat{J}$ in the iterative process in Algorithm \ref{alg:sf_iter}. In Section \ref{sec:cur_thms}, we will show that $\mathbf{W^*} = \argmin_{\mathbf{W} \in \mathbb{R}^{n \times m}} J(\mathbf{W})$, where $\mathbf{W^*}$ is the output of Algorithm \ref{alg:sf_iter}.

\begin{algorithm}
\caption{Convex Optimization Using the Surrogate Functional}
\begin{algorithmic}[1]
\Require $\mathbf{X} \in \mathbb{R}^{m \times n}, \lambda_C \geq 0, \mu > ||\mathbf{X} ||_2^4$ 
\Ensure $\mathbf{W^*} \in \mathbb{R}^{n \times m}$ s.t. $\mathbf{W^*} = \argmin_{\mathbf{W} \in \mathbb{R}^{n \times m}} ||\mathbf{X-XWX}||_F^2 + \lambda_C \sum_{i=1}^n  \|\mathbf{W}(i,:) \|_{\infty}$
\State $\mathbf{W}^0 = \mathbf{0}$
\State $k = 0$
\Repeat  \Comment each iteration solves $\mathbf{W^k} = \argmin_{\mathbf{Y} \in \mathbb{R}^{n \times m}} \widehat{J}(\mathbf{Y},\mathbf{W^{k-1}})$.
    \State $k = k+1$
    \State $\mathbf{L} =  \mu \mathbf{(W^{k-1})^T} + \mathbf{XX^TX} - \mathbf{XX^T (W^{k-1})^T X^TX}$
    \For{$i = 1 \: \text{to} \: n$}
        \State $\mathbf{W^k}(i,:) = \argmin_{\mathbf{y} \in \mathbb{R}^{m}} \left[ \frac{1}{2}||\mathbf{y} - \frac{1}{\mu}\mathbf{L^T}(i,:)||^2_2 + \frac{\lambda_C}{2\mu} ||\mathbf{y}||_{\infty} \right]$
    \EndFor
\Until {convergence of the sequence $\{ \mathbf{W^k} \}$}
\State $\mathbf{W^*} = \mathbf{W^k}$  
\Comment{$\mathbf{W^*} = \lim_{k \rightarrow \infty} \{ \mathbf{W^k} \}$}
\State \Return $\mathbf{W^*}$
\end{algorithmic}
\label{alg:sf_iter}
\end{algorithm}

\subsection{Complexity}

For this analysis we will assume that $r \ll m$ and $c \ll n$, and rename $\mu$ as $\mu_C$ to avoid confusion with $\mu_R$, a similar parameter needed to solve the line 19 minimization. We provide computational complexities in Table \ref{tab:cc} that help determine the overall complexity of Algorithm \ref{alg:cur_co}.  Quantities that are used in each iteration of one of the bisection loops should be computed once before the loop begins. For the first bisection loop on lines 3-13, this includes $\mathbf{XX^TX}$ (which is the transpose of the matrix computed to find $\lambda_C^*$ in line 1), $\mathbf{XX^T}$, $\mathbf{X^TX}$, and $\mu_C > \| \mathbf{X} \|_2^4$ which is an input to Algorithm \ref{alg:sf_iter}. For the second bisection loop on lines 17-27, this includes $\mathbf{XX^TC}$ (which is the transpose of the matrix computed to find $\lambda_R^*$ in line 15), $\mathbf{XX^T}$ (which was already computed before the first bisection loop), $\mathbf{C^TC}$, and $\mu_R > \| \mathbf{X} \|_2^2 \| \mathbf{C} \|_2^2$ (an input to the algorithm that solves the line 19 minimization), in which $\| \mathbf{X} \|_2$ was already computed as well. 

\begin{table}[h]
    \centering
    \rowcolors{2}{gray!25}{white}
    \begin{tabular}{|l|l|}
        \hline
        \textbf{Computation} & \textbf{Complexity} \\
        \hline
         $\lambda_C^*$ & $O(mn^2)$ if $m \geq n$ or $O(m^2n)$ if $m < n$  \\  
         $\mathbf{XX^TX}$ & $O(mn)$ utilizing $\lambda_C^*$ computation \\
         $\mathbf{XX^T}$ & $O(m^2n)$ \\
         $\mathbf{X^TX}$ & $O(mn^2)$ \\
         $\mu_C$ & $O(mn)$ \\
         $\lambda_R^*$ &  $O(cmn)$ if $m \geq c$ or $O(cm^2)$ (utilizing $\mathbf{XX^T}$ computation) if $m < c$  \\
         $\mathbf{XX^TC}$ & $O(cm)$ utilizing $\lambda_R^*$ computation\\
         $\mathbf{C^TC}$ & $O(c^2m)$ \\
         $\mu_R$ & $O(cm)$ \\         
        \hline
    \end{tabular}
    \caption{Computational complexities of quantities used in Algorithm \ref{alg:cur_co}.}
    \label{tab:cc}
\end{table}

Before we analyze the entirety of Algorithm \ref{alg:cur_co}, we will analyze the complexity of Algorithm \ref{alg:sf_iter}, which solves the minimization problem on line 5 of Algorithm \ref{alg:cur_co}.  In each iteration, the most expensive steps are the computation of $\mathbf{L}$ and the $n$ proximal operators.  $\mathbf{L}$ can be computed in $O(mn(m+n))$ time and each proximal operator can be computed in $O(m \log m)$ time.  Determining if the sequence has converged in line 9 can be completed in $O(mn)$ time.  Hence the total time for Algorithm \ref{alg:sf_iter}, assuming $k$ iterations are completed, is $O(kmn(m+n))$.  A similar analysis shows that the complexity of the algorithm for solving the minimization problem on line 19 of Algorithm \ref{alg:cur_co} is $O(\hat{k}cm(m+c))$, assuming $\hat{k}$ iterations are completed.  For each minimization problem, we have found that using a small maximum number of iterations in practice, e.g., 20, is sufficient for use in the CUR algorithm, Algorithm \ref{alg:cur_co}. For the remainder of this analysis, we will assume that the number of iterations for each minimization problem is a small constant.

In Algorithm \ref{alg:cur_co}, the bisection loop in lines 3-13 dominates the computational complexity.  This loop runs in $O(\ell mn(m+n))$ time, assuming $\ell$ iterations.  This is due to the computation time of Algorithm \ref{alg:sf_iter} and that finding the set $I_C$ can be completed in $O(mn)$ time.  Using a similar analysis, the second bisection loop on lines 17-27 is an order of magnitude less expensive, i.e., $O(\hat{\ell} cm(m+c))$ assuming $\hat{\ell}$ iterations. The computation of $\mathbf{U}$ involves two pseudoinverses, $\mathbf{C^+}$ and $\mathbf{R^+}$, which can be computed in $O(cm \min(c,m))$ and $O(nr \min(n,r))$ time, respectively.  The product $\mathbf{U = C^+XR^+}$ can be computed in $O(cn(m+r))$ time if $m \geq n$, or $O(mr(n+c))$ time if $m < n$.  Hence the total computational complexity for Algorithm \ref{alg:cur_co} is $O(\ell mn(m+n))$.

\subsection{Generalizations}
\label{sec:cur_gen}

To form the matrix $\mathbf{C}$, our algorithm and implementation can also accommodate objective functions of the form 
\[
||\mathbf{X-XWX}||_F^2 + \lambda_C \sum_{i=1}^n ||\mathbf{W}(i,:)||_p^p,
\]
for $1 \leq p \leq 2$.  The theory for using the surrogate functional technique with these choices of objective functions is already complete \cite{daubechies_etal_2004}, and the only change to the implementation detailed above is that the proximal operator in Algorithm \ref{alg:sf_iter} would be of the $\ell_p$-norm.  Closed form solutions for the proximal operator of the $\ell_1$ and $\ell_2$ norms exist, making these easy choices to implement.  Similar penalty function adaptations can be made to the objective function used to form the matrix $\mathbf{R}$.  Hence, our algorithm and implementation provide a CUR framework.  We also note that the objective function could be generalized as
\[
\min_{\mathbf{W} \in \mathbb{R}^{n \times m}} ||\mathbf{X-XWX}||_p^p + \lambda_C \sum_{i=1}^n ||\mathbf{W}(i,:)||_{\infty},
\]
where $p \in [1, \infty]$, which remains a potential area for future work.

\section{Theoretical Foundation}
\label{sec:cur_thms}

In this section, we follow the theoretical approach of \cite{daubechies_etal_2004} to prove the correctness of Algorithm \ref{alg:sf_iter} for solving the minimization on line 5 of Algorithm \ref{alg:cur_co}, i.e., that 
\[
\mathbf{W^*} = \argmin_{\mathbf{W} \in \mathbb{R}^{n \times m}} ||\mathbf{X-XWX}||_F^2 + \lambda_C \sum_{i=1}^n  \|\mathbf{W}(i,:) \|_{\infty},
\]
where $\mathbf{W^*}$ is the output of Algorithm \ref{alg:sf_iter}.  
The correctness of the algorithm for solving the minimization on line 19 of Algorithm \ref{alg:cur_co}, can be proved similarly.  

For constant $\mathbf{X}, \mu, \lambda_C$, let the nonlinear operator $\mathbf{T}_{\mathbf{X}, \mu, \lambda_C}: \mathbb{R}^{n \times m} \rightarrow \mathbb{R}^{n \times m}$ be defined as $\mathbf{Z} \mapsto \mathbf{T}_{\mathbf{X}, \mu, \lambda_C}(\mathbf{Z})=\mathbf{W}$ and given by:
\begin{enumerate}
    \item construct $\mathbf{L(Z)} = \mu \mathbf{Z^T} + \mathbf{XX^TX} - \mathbf{XX^TZ^TX^TX}$,
    \item for each row of $\mathbf{L^T}$ (i.e., $\forall i \in [n]$), solve 
    \[
        \mathbf{W}(i,:) = \argmin_{\mathbf{y} \in \mathbb{R}^m} \left[\frac{1}{2}||\mathbf{y}-\frac{1}{\mu}\mathbf{L^T}(i,:)||^2_2 + \frac{\lambda_C}{2\mu}||\mathbf{y}||_\infty\right],     
    \]
    and
    \item reassemble $\mathbf{W}$ from its rows,
\end{enumerate}
so that we can write each $\mathbf{W}^k$ produced in Algorithm \ref{alg:sf_iter} as $\mathbf{W}^k = \mathbf{T}^k_{\mathbf{X}, \mu, \lambda_C}(\mathbf{W}^0)$.  
The primary result that will we prove in this section is that the sequence $\{ \mathbf{W}^k \}_{k \in \mathbb{N}}$ converges to a minimizer of $||\mathbf{X-XWX}||_F^2 + \lambda_C \sum_{i=1}^n  \|\mathbf{W}(i,:) \|_{\infty}$.  This is formally stated below.

\begin{theorem}
\label{thm:alg2_correct}
Let $\lambda_C \in \mathbb{R} \geq 0$; $\mu \in \mathbb{R} > 0$; $\mathbf{X} \in \mathbb{R}^{m \times n}$; and $\mathbf{W}, \mathbf{Z} \in \mathbb{R}^{n \times m}$.  Define 
\begin{equation*}
J(\mathbf{W}) = ||\mathbf{X-XWX}||_F^2 + \lambda_C \sum_{i=1}^n \| \mathbf{W}(i,:) \|_{\infty}, 
\end{equation*}
and the nonlinear operator $\mathbf{T}_{\mathbf{X}, \mu, \lambda_C}: \mathbb{R}^{n \times m} \rightarrow \mathbb{R}^{n \times m}$ as $\mathbf{Z} \mapsto \mathbf{T}_{\mathbf{X}, \mu, \lambda_C}\mathbf{(Z)=W}$ as above.

\begin{enumerate}
    \item[a.] If $\mu > ||\mathbf{X}||_2^4$, then the sequence $\{\mathbf{W}^k = \mathbf{T}^k_{\mathbf{X}, \mu, \lambda_C}(\mathbf{W}^0)\}_{k \in \mathbb{N}}$ produced by Algorithm \ref{alg:sf_iter} converges to a fixed point of $\mathbf{T}_{\mathbf{X}, \mu, \lambda_C}$.
    \item[b.] A fixed point of $\mathbf{T}_{\mathbf{X}, \mu, \lambda_C}$ is a minimizer of $J(\mathbf{W})$.
    \item[c.] $J(\mathbf{W})$ has a unique minimizer if $\mathbf{X}$ is square and full rank.
\end{enumerate}
\end{theorem}
\noindent To condense notation for the remainder of this section, we write $\mathbf{T}_{\mathbf{X}, \mu, \lambda_C} = \mathbf{T}$.

\subsection{Convergence to a Fixed Point of $\mathbf{T}$} 

We first provide six lemmas to assist in the proof of Theorem \ref{thm:alg2_correct}, part a.

\begin{lemma}
\label{lemma:psd}
    $\mathbf{(X^TX) \otimes (XX^T)}$ is symmetric and positive semidefinite. 
\end{lemma}
\begin{proof}
    $\mathbf{(X^TX) \otimes (XX^T)}$ is clearly symmetric. We will show that $\langle (\mathbf{(X^TX) \otimes (XX^T)})\mathbf{v}, \mathbf{v} \rangle \geq 0$, $\forall \mathbf{v} \in \mathbb{R}^{mn}$, thus proving the lemma. Fix $\mathbf{v} \in \mathbb{R}^{mn}$ and $\mathbf{V} \in \mathbb{R}^{m \times n}$ such that $\text{vec}(\mathbf{V}) = \mathbf{v}$.
    \begin{equation*}
        \begin{split}
            \langle (\mathbf{(X^TX) \otimes (XX^T)})\mathbf{v}, \mathbf{v} \rangle & = \langle \text{vec}(\mathbf{(XX^T)V(X^TX)}), \text{vec}(\mathbf{V}) \rangle \\
            & = \langle \mathbf{(XX^T)V(X^TX)}, \mathbf{V} \rangle_{\text{F}} \\
            & = \tr\{\mathbf{(XX^T)V(X^TX)V^T}\} \\
            & = \tr\{\mathbf{(XX^T)}^{\frac{1}{2}}\mathbf{V(X^TX)}^{\frac{1}{2}}\mathbf{(X^TX)}^{\frac{1}{2}}\mathbf{V^T(XX^T)}^{\frac{1}{2}}\} \\   
            & = ||\mathbf{(XX^T)}^{\frac{1}{2}}\mathbf{V(X^TX)}^{\frac{1}{2}}||_F^2 \geq 0,
        \end{split}
    \end{equation*}
where $\langle . \, , . \rangle_{\text{F}}$ is the Frobenius inner product. 

\end{proof}

\begin{lemma}
\label{lemma:sigma_bound}
$||\mathbf{(X^TX) \otimes (XX^T)}||_2 \leq (\sigma_{\text{max}}(\mathbf{X}))^4 = ||\mathbf{X}||_2^4$, where $\sigma_{\text{max}}(\mathbf{X})$ is the largest singular value of the matrix $\mathbf{X}$.    
\end{lemma}

\begin{proof}
By Lemma \ref{lemma:psd}, the matrix $\mathbf{(X^TX) \otimes (XX^T)}$ is symmetric and positive semidefinite.  Thus, the eigenvalues and singular values of $\mathbf{(X^TX) \otimes (XX^T)}$ are the same. Hence
\begin{equation*}
    \begin{split}
        ||\mathbf{(X^TX) \otimes (XX^T)}||_2 & = \sigma_{\text{max}}(\mathbf{(X^TX) \otimes (XX^T))} \\
        & = \max_{||\mathbf{v}||_2=1} \langle \mathbf{((X^TX) \otimes (XX^T))v, v}\rangle \\
        & = \max_{||\mathbf{V}||_F=1} ||\mathbf{(XX^T)}^{\frac{1}{2}}\mathbf{V(X^TX)}^{\frac{1}{2}}||_F^2,
    \end{split}
\end{equation*}
where the last equation follows from the proof of Lemma \ref{lemma:psd}.
\begin{equation*}
    \begin{split}
        ||\mathbf{(XX^T)}^{\frac{1}{2}}\mathbf{V(X^TX)}^{\frac{1}{2}}||_F & \leq ||\mathbf{(XX^T)}^{\frac{1}{2}}||_2 ||\mathbf{V(X^TX)}^{\frac{1}{2}}||_F \\
        & = ||\mathbf{(XX^T)}^{\frac{1}{2}}||_2 ||\mathbf{(X^TX)}^{\frac{1}{2}}\mathbf{V^T}||_F \\
        & \leq ||\mathbf{(XX^T)}^{\frac{1}{2}}||_2 ||\mathbf{(X^TX)}^{\frac{1}{2}}||_2 ||\mathbf{V^T}||_F \\
        & \leq \sigma_{\text{max}}(\mathbf{(XX^T)}^{\frac{1}{2}}) \sigma_{\text{max}}(\mathbf{(X^TX)}^{\frac{1}{2}}) ||\mathbf{V^T}||_F \\
        & = (\sigma_{\text{max}}(\mathbf{X}))^2 ||\mathbf{V^T}||_F. 
    \end{split}
\end{equation*}
Thus 
\[
\max_{||\mathbf{V}||_F=1} ||\mathbf{(XX^T)}^{\frac{1}{2}}\mathbf{V(X^TX)}^{\frac{1}{2}}||_F^2 \leq (\sigma_{\text{max}}\mathbf{(X)})^4,
\]
proving the result.
\end{proof}

\begin{lemma}
\label{lem:L_opnorm}
    Let  $\mu \geq ||\mathbf{X}||_2^4$ and define  the operator $\mathscr{L}: \mathbf{U} \mapsto \mu \mathbf{U} - \mathbf{XX^T U X^TX}$.  Then $ || \mathscr{L} || \leq \mu$.  
\end{lemma}

\begin{proof}
The operator norm of $\mathscr{L}$ is
\begin{equation*}
    \begin{split}
        || \mathscr{L} || & = \max_{||\mathbf{U}||_F = 1} || \mathscr{L}(\mathbf{U})||_F \\
        & = \max_{||\mathbf{U}||_F = 1} ||\text{vec}(\mu \mathbf{U} - \mathbf{XX^T U X^TX})||_ 2\\
        & = \max_{||\mathbf{U}||_F = 1} ||(\mu \mathbf{I} - \mathbf{(X^TX) \otimes (XX^T))}\text{vec}(\mathbf{U})||_2 \\  
        & = \sigma_{\text{max}}(\mu \mathbf{I} - \mathbf{(X^TX) \otimes (XX^T)}). 
    \end{split}
\end{equation*}
By Lemmas \ref{lemma:psd} and \ref{lemma:sigma_bound}, $\mathbf{(X^TX) \otimes (XX^T)}$ is a symmetric positive semidefinite matrix with 2-norm bounded above by $||\mathbf{X}||_2^4$.  Hence 
\[
  || \mathscr{L} || = \sigma_{\text{max}}(\mu \mathbf{I} - \mathbf{(X^TX) \otimes (XX^T)}) \leq \mu.
\]    
\end{proof}

\begin{lemma}
\label{lem:nonexp}
Let $\mu \geq ||\mathbf{X}||_2^4$.  Then the mapping $\mathbf{T}$ is nonexpansive, i.e., $\forall \, \mathbf{Z}, \mathbf{Z'} \in \mathbb{R}^{n \times m}$,
\[
||\mathbf{T(Z)}-\mathbf{T(Z')}||_{F} \leq ||\mathbf{Z}-\mathbf{Z'}||_{F}.
\]
\end{lemma}

\begin{proof}

By the fact that $\textbf{prox}_{\frac{\lambda}{2\mu} ||\cdot||_\infty}$ is nonexpansive \cite{parikh_boyd_2014} and Lemma \ref{lem:L_opnorm}, we have 
\begin{equation*}
    \begin{split}
        ||\mathbf{T(Z)}-\mathbf{T(Z')}||^2_{F} & = \sum_{i=1}^n \left\| \textbf{prox}_{\frac{\lambda}{2\mu} ||\cdot||_\infty} \left(\frac{1}{\mu}\mathbf{[L(Z)]^T}(i,:) \right) -  \textbf{prox}_{\frac{\lambda}{2\mu} ||\cdot||_\infty} \left( \frac{1}{\mu}\mathbf{[L(Z')]^T}(i,:) \right) \right \|_2^2  \\
        & \leq \sum_{i=1}^n \left\| \frac{1}{\mu}\mathbf{[L(Z)]^T}(i,:) - \frac{1}{\mu}\mathbf{[L(Z')]^T}(i,:) \right \|_2^2 \\
        & = \frac{1}{\mu^2} ||\mathbf{L(Z)}-\mathbf{L(Z')}||^2_{F} \\
        & = \frac{1}{\mu^2} ||\mathscr{L}(\mathbf{Z - Z'}) ||^2_{F} \\
        & \leq  \frac{1}{\mu^2} ||\mathscr{L}||^2 ||\mathbf{Z} - \mathbf{Z'}||_F^2 \\
        & \leq ||\mathbf{Z} - \mathbf{Z'}||_F^2.
    \end{split}
\end{equation*}

\end{proof}

We omit proofs of the next lemma and theorem, as they largely mirror those of \cite{daubechies_etal_2004}.

\begin{lemma}
\label{thm:asy_reg_maintext}
    Let $\mu > ||\mathbf{X}||_2^4$.  Then the mapping $\mathbf{T}$ is asymptotically regular, i.e. $\forall \, \mathbf{Z} \in \mathbb{R}^{n \times m}$,
\[
\lim_{k \rightarrow \infty} ||\mathbf{T}^{k+1}(\mathbf{Z})-\mathbf{T}^{k}(\mathbf{Z})||_{F} = 0.
\]
\end{lemma}

\begin{lemma}
\label{thm:conv_maintext}
    Let the mapping $\mathbf{T}: \mathbb{R}^{n \times m} \rightarrow \mathbb{R}^{n \times m}$ be nonexpansive and asymptotically regular.  Then the sequence $\{\mathbf{W}^k = \mathbf{T}^k(\mathbf{W}^0)\}_{k \in \mathbb{N}}$ converges to a fixed point of $\mathbf{T}$.
\end{lemma}

We are now ready to prove Theorem \ref{thm:alg2_correct}, part a.

\begin{proof}[Proof of Theorem \ref{thm:alg2_correct}, part a.]
    By Lemmas \ref{lem:nonexp} and \ref{thm:asy_reg_maintext}, $\mathbf{T}$ is nonexpansive and asymptotically regular.  By Lemma \ref{thm:conv_maintext}, the result is proven.
\end{proof}

\subsection{Convergence to a Minimizer of $J$}

We now focus on proving Theorem \ref{thm:alg2_correct}, part b and begin by establishing three lemmas.

\begin{lemma}
\label{lemma:conv}
Let $f:I \rightarrow \mathbb{R}$ be convex, with $I=(a,b)$. If  $f(x) = f_1(x) + \alpha x^2$, where $\alpha > 0$ and $f_1(x)$ is piecewise linear, then $f_1$ is convex and therefore $f$ is strictly convex.
\end{lemma}

\begin{proof}
Let $x_0, x_1, x_2, ..., x_n,x_{n+1} \in I$ such that $a=x_0 < x_1 < x_2 < ... < x_n < x_{n+1}=b$, and $\forall k \in [n+1]$, $f_1$ restricted to $(x_{k-1},x_{k})$ is linear and given by $f_1(x) = a_k x + b_k$.  Due to the convexity of $f$, $f_1$ is continuous.  Hence $a_k x_{k} + b_k = a_{k+1} x_{k}  + b_{k+1}$, for every $k \in [n]$.

\noindent We show that $a_1 \leq a_2 \leq ... \leq a_{n+1}$, which implies that $f_1$ is convex.  Consider the inequality $a_k \leq a_{k+1}$ for any $k \in [n]$.  Let $h \in \mathbb{R}$ such that $0 \leq h < \min (x_k-x_{k-1},x_{k+1}-x_k)$. The convexity of $f$ implies that
\[
\frac{f(x_k-h)+f(x_k+h)}{2} \geq f(x_k).
\]
Thus,
\begin{equation*}
    \begin{split}
        & \frac{a_k(x_k-h) + b_k + \alpha(x_k-h)^2 + a_{k+1}(x_k+h) + b_{k+1} + \alpha(x_k +h)^2}{2} \geq \\
        & \quad \frac{a_k x_k+b_k + \alpha x_k^2 + a_{k+1} x_k+b_{k+1} + \alpha x_k^2}{2},
    \end{split}
\end{equation*}
where we have used the continuity of $f_1$ at $x_k$ on the right hand side of the inequality. Hence
\begin{equation}
\label{eqn:quad_h}
(a_{k+1}-a_{k})h + 2 \alpha h^2 \geq 0.   
\end{equation}
$(a_{k+1}-a_{k})h + 2 \alpha h^2$ is a convex quadratic in $h$ with roots $0$ and $\frac{-(a_{k+1}-a_{k})}{2\alpha}$.  This leads to three cases:
\begin{enumerate}
    \item The root $\frac{-(a_{k+1}-a_{k})}{2\alpha} = 0$, i.e. $0$ is a root of multiplicity two.  This implies $a_{k} = a_{k+1}$ and Equation \ref{eqn:quad_h} holds.
    \item The root $\frac{-(a_{k+1}-a_{k})}{2\alpha} < 0$.  This implies $a_{k} < a_{k+1}$ and Equation \ref{eqn:quad_h} holds for $h \geq 0$.  
    \item The root $\frac{-(a_{k+1}-a_{k})}{2\alpha} > 0$. This implies $a_{k} > a_{k+1}$.  For $h$ such that \newline $0 < h < \min(x_{k}-x_{k-1},x_{k+1}-x_{k}, \frac{-(a_{k+1}-a_{k})}{2\alpha})$, Equation \ref{eqn:quad_h} does not hold. 
\end{enumerate}
Thus, Equation \ref{eqn:quad_h} holds $\forall h$ such that $0 \leq h < \min(x_k-x_{k-1},x_{k+1}-x_k)$ if and only if $a_{k} \leq a_{k+1}$.  Hence $a_1 \leq a_2 \leq ... \leq a_{n+1}$, which implies that $f_1$ is convex.
\end{proof}

\begin{lemma}
\label{lem:F_greater_0}
    Let $f:I \rightarrow \mathbb{R}$ with $I=(a,b)$ such that $0 \in I$ be defined as $f(\alpha) = F(\alpha) + \frac{1}{2}\alpha^2 \geq 0$.  If $F$ is continuous, piecewise linear, convex, and $F(0)=0$, then $F(\alpha) \geq 0$.
\end{lemma}

\begin{proof}
Let $\alpha_0, \alpha_1, \alpha_2, ..., \alpha_n,\alpha_{n+1} \in I$ such that $a=\alpha_0 < \alpha_1 < \alpha_2 < ... < \alpha_n < \alpha_{n+1}=b$, and $\forall k \in [n+1]$, $F(\alpha) = a_k  \alpha  + b_k$, if $\alpha_{k-1} \leq \alpha \leq \alpha_{k}$.
 We will use two cases to prove the result. \newline
 Case 1: Suppose $0 \in (\alpha_{k-1}, \alpha_{k})$ for some $k \in [n+1]$.  Since $F(0) = 0$, $b_k = 0$.  In addition, $a_k = 0$, which we will show by contradiction.  Suppose $a_k > 0$ and $\alpha_{k-1} < \epsilon < 0$ such that $|\epsilon| < 2a_k$.  Then $f(\epsilon) = a_k \epsilon + \frac{1}{2}\epsilon^2 < 0$, which is a contradiction.  The case in which $a_k < 0$ similarly leads to a contradiction.  Hence $F(\alpha) = 0$ on [$\alpha_{k-1}, \alpha_{k}]$.  Due to the convexity of $F$, $a_1 \leq a_2 \leq \cdots \leq a_n \leq a_{n+1}$. Therefore $F'(\alpha) \leq 0$ for $\alpha \leq \alpha_{k-1}$ and $F'(\alpha) \geq 0$ for $\alpha \geq \alpha_{k}$.  Hence $F(\alpha) \geq 0$. \newline
 Case 2: Suppose $\alpha_k = 0$ for some $k \in [n]$.  Since $F(0)=0$, $F(\alpha) = a_{k} \alpha $ and $f(\alpha) = a_{k} \alpha  + \frac{1}{2}\alpha^2$ on $\alpha_{k-1} \leq \alpha \leq \alpha_{k} = 0$.  We will show that $a_{k} \leq 0$.  Suppose $a_{k} > 0$.  Then $f(\alpha) = \alpha(a_{k} + \frac{1}{2}\alpha) < 0$ on $-2a_{k} < \alpha < 0$, which is a contradiction.  Similarly, $F(\alpha) = \alpha a_{k+1}$ and $f(\alpha) =  \alpha a_{k+1} + \frac{1}{2}\alpha^2$ on $0 = \alpha_{k} \leq \alpha \leq \alpha_{k+1}$.  We will show that $a_{k+1} \geq 0$.  Suppose $a_{k+1} < 0$.  Then $f(\alpha) = \alpha(a_{k+1} + \frac{1}{2}\alpha) < 0$ on $0 < \alpha < 2|a_{k+1}|$, which is a contradiction.  Due to the convexity of $F$, $a_1 \leq a_2 \leq \cdots \leq a_n \leq a_{n+1}$. Therefore $F'(\alpha) \leq a_{k} \leq 0$ for $\alpha \leq \alpha_{k}$ and $F'(\alpha) \geq a_{k+1} \geq 0$ for $\alpha \geq \alpha_{k}$.  Hence $F(\alpha) \geq 0$. 
    
\end{proof}

\begin{lemma}
\label{lem:min}
Let $\mathbf{X} \in \mathbb{R}^{m \times n}$, $\mathbf{W}, \mathbf{Z} \in  \mathbb{R}^{n \times m}$, $\mu > 0$, $\lambda_C \geq 0$, and $\forall i \in [n]$
\[
\mathbf{W}(i,:) = \argmin_{\mathbf{y} \in \mathbb{R}^{m}} 
\left[ \frac{1}{2} \|\mathbf{y} - \frac{1}{\mu}\mathbf{L^T}(i,:) \|^2_2 + \frac{\lambda_C}{2\mu}\|\mathbf{y}\|_{\infty} \right],
\]
where $\mathbf{L(Z)} = \mu \mathbf{Z^T} + \mathbf{XX^TX} - \mathbf{XX^TZ^TX^TX}$.
Then for every $\mathbf{H} \in  \mathbb{R}^{n \times m}$,
\[
\widehat{J}(\mathbf{W+H};\mathbf{Z}) \geq \widehat{J}(\mathbf{W};\mathbf{Z}) + \mu||\mathbf{H}||_F^2.
\]
\end{lemma}

\begin{proof}
By definition, 
\begin{equation*}
\begin{split}
& \widehat{J}(\mathbf{W+H}, \mathbf{Z})  = ||\mathbf{X-X(W+H)X}||_F^2 + \lambda_C \sum_{i=1}^n \| (\mathbf{W+H})(i,:) \|_{\infty} \\ 
& \quad \quad + \mu ||\mathbf{W+H}-\mathbf{Z}||_F^2 - ||\mathbf{X(W+H)X}-\mathbf{XZX}||_F^2 \\
& = \widehat{J}(\mathbf{W}, \mathbf{Z}) + \lambda_C \sum_{i=1}^n \left(||(\mathbf{W+H})(i,:)||_\infty - ||\mathbf{W}(i,:)||_\infty \right)  + 2 \tr\{\mathbf{H}(\mu \mathbf{W^T - L})\} + \mu ||\mathbf{H}||_F^2 \\
& = \widehat{J}(\mathbf{W}, \mathbf{Z}) + \sum_{i=1}^n \left[ \lambda_C(||(\mathbf{W+H})(i,:)||_\infty - ||\mathbf{W}(i,:)||_\infty) + 2\mu \left\langle \mathbf{H}(i,:), (\mathbf{W} - \frac{1}{\mu}\mathbf{L^T})(i,:) \right\rangle \right] \\
& \quad \quad + \mu ||\mathbf{H}||_F^2. 
\end{split}
\end{equation*}
To prove the result we need to show that
\[
\sum_{i=1}^n \left[ \lambda_C(||(\mathbf{W+H})(i,:)||_\infty - ||\mathbf{W}(i,:)||_\infty) + 2\mu \left\langle \mathbf{H}(i,:), (\mathbf{W} - \frac{1}{\mu}\mathbf{L^T})(i,:) \right\rangle \right] \geq 0.
\]
Hence, it suffices to show that $\forall i \in [n]$,
\begin{equation}
\label{pf:F}
\frac{\lambda_C}{2\mu} (||(\mathbf{W+H})(i,:)||_\infty - ||\mathbf{W}(i,:)||_\infty) + \left\langle \mathbf{H}(i,:), (\mathbf{W} - \frac{1}{\mu}\mathbf{L^T})(i,:) \right\rangle \geq 0.
\end{equation}

\noindent To simplify notation, let $\mathbf{w} = \mathbf{W}(i,:)$, $\mathbf{h} = \mathbf{H}(i,:)$ and $\boldsymbol\ell = \mathbf{L^T}(i,:)$.  Let
\[
F(\mathbf{h}) = \frac{\lambda_C}{2\mu} (||\mathbf{w+h}||_\infty - ||\mathbf{w}||_\infty) + \left\langle \mathbf{h}, \mathbf{w} - \frac{1}{\mu}\boldsymbol  \ell \right\rangle, 
\]
where $\mathbf{h} = \alpha \mathbf{u}$ and $\mathbf{u} \in \mathbb{R}^{m}$ is a random unit vector.  Then we can denote $F(\mathbf{h}) = F(\alpha, \mathbf{u})$, and fix a $\mathbf{u}$ so that $F$ is only a function of $\alpha$: 
\[
F(\alpha) = \frac{\lambda_C}{2\mu} (||\mathbf{w+\alpha\mathbf{u}}||_\infty - ||\mathbf{w}||_\infty) + \alpha \left\langle \mathbf{u}, \mathbf{w} - \frac{1}{\mu}\boldsymbol \ell \right\rangle. 
\]
Let $G(\mathbf{w}) = \frac{1}{2} ||\mathbf{w} - \frac{1}{\mu} \boldsymbol \ell||^2_2 + \frac{\lambda_C}{2\mu}||\mathbf{w}||_{\infty}$. Then
\[
G(\mathbf{w} + \alpha \mathbf{u}) = \frac{1}{2}||\mathbf{w} + \alpha \mathbf{u} - \frac{1}{\mu} \boldsymbol \ell||^2_2 + \frac{\lambda_C}{2\mu}||\mathbf{w} + \alpha \mathbf{u}||_{\infty},
\]
and
\begin{equation*}
\begin{split}
G(\mathbf{w} + \alpha \mathbf{u}) - G(\mathbf{w}) & = \frac{1}{2}||\mathbf{w} + \alpha \mathbf{u} - \frac{1}{\mu} \boldsymbol \ell||^2_2 - \frac{1}{2}||\mathbf{w} - \frac{1}{\mu} \boldsymbol \ell||^2_2 + \frac{\lambda_C}{2\mu}||\mathbf{w} + \alpha \mathbf{u}||_{\infty} - \frac{\lambda_C}{2\mu}||\mathbf{w}||_{\infty} \\
& = \frac{1}{2} \left( ||\mathbf{w} - \frac{1}{\mu} \boldsymbol \ell||^2_2 + 2 \left\langle \alpha \mathbf{u}, \mathbf{w} - \frac{1}{\mu} \boldsymbol \ell \right\rangle + ||\alpha \mathbf{u}||^2_2 - ||\mathbf{w} - \frac{1}{\mu} \boldsymbol \ell||^2_2 \right) \\ 
& \quad + \frac{\lambda_C}{2\mu}(||\mathbf{w} + \alpha \mathbf{u}||_{\infty} - ||\mathbf{w}||_{\infty}) \\
& = \frac{\lambda_C}{2\mu}(||\mathbf{w} + \alpha \mathbf{u}||_{\infty} - ||\mathbf{w}||_{\infty}) + \alpha \left\langle \mathbf{u}, \mathbf{w} - \frac{1}{\mu} \boldsymbol \ell \right\rangle + \frac{1}{2}\alpha^2.
\end{split}
\end{equation*}

\noindent Thus, $F(\alpha) = G(\mathbf{w} + \alpha \mathbf{u}) - G(\mathbf{w}) - \frac{1}{2}\alpha^2$.  Let  $f(\alpha) = F(\alpha) + \frac{1}{2}\alpha^2$.  We note that $f$ is convex since $f(\alpha) = G(\mathbf{w} + \alpha \mathbf{u}) - G(\mathbf{w})$.  In order to use Lemma \ref{lemma:conv}, we also need to show that $F(\alpha)$ is piecewise linear in $\alpha$.  There is a constant term of $F(\alpha)$, $ \frac{-\lambda_C}{2\mu}|| \mathbf{w}||_\infty$, and a linear term, $\alpha \langle \mathbf{u}, \mathbf{w} - \frac{1}{\mu}\boldsymbol  \ell \rangle$.  The remaining term, $\frac{\lambda_C}{2\mu} ||\mathbf{w+\alpha\mathbf{u}}||_\infty$ is piecewise linear in $\alpha$, since as $\alpha$ increases
\[
||\mathbf{w} + \alpha \mathbf{u}||_\infty = \max(\mathbf{w}_1 + \alpha \mathbf{u}_1, ... \:, \mathbf{w}_m + \alpha \mathbf{u}_m, -\mathbf{w}_1 - \alpha \mathbf{u}_1, ... \:, -\mathbf{w}_m - \alpha \mathbf{u}_m), 
\]
and the maximum of a set of linear functions is piecewise linear.  Thus, $F(\alpha)$ is piecewise linear, and by Lemma \ref{lemma:conv}, $F(\alpha)$ is convex.  

The remaining step is to show that $F(\alpha) \geq 0$, which will establish the claim in Equation \ref{pf:F} and thus the result. 
 Since $\mathbf{w} = \argmin_{\mathbf{y}} G(\mathbf{y})$, we have $f(\alpha) = G(\mathbf{w} + \alpha \mathbf{u}) - G(\mathbf{w}) \geq 0$. We also know that $F(\alpha)$ is continuous and piecewise linear in $\alpha$, convex, and $F(0)=0$.  Hence by Lemma \ref{lem:F_greater_0}, $F(\alpha) \geq 0$.

\end{proof}

Lemma \ref{lem:min} is applied with $\mathbf{W}$ equal to a fixed point of $\mathbf{T}$ to prove Theorem \ref{thm:alg2_correct}, part b.  However, the proof mirrors that of \cite{daubechies_etal_2004}, so is omitted.  To complete the proof of Theorem \ref{thm:alg2_correct}, we provide the proof of part c.

\begin{proof}[Proof of Theorem \ref{thm:alg2_correct}, part c.]
The minimizer of $J(\mathbf{W})$ is unique if $\| \mathbf{X} - \mathbf{XWX} \|_F^2$ is strictly convex in $\mathbf{W}$.  Since
\[
\| \mathbf{X} - \mathbf{XWX} \|_F^2 = ||\mathbf{(X^T \otimes X) w}-\mathbf{b}||^2_2, 
\]
where $\mathbf{b} = \text{vec}(\mathbf{X}) \in \mathbb{R}^{mn}$ and $\mathbf{w} = \text{vec}(\mathbf{W}) \in \mathbb{R}^{mn}$, we need to guarantee that $\mathbf{(X^T \otimes X)}$ has full rank.  Since $\text{rank}\mathbf{(X^T \otimes X)} = \text{rank}(\mathbf{X}^T)\text{rank}(\mathbf{X}) = (\text{rank}(\mathbf{X}))^2$, $\mathbf{(X^T \otimes X)}$ has full rank if $\mathbf{X}$ is square and full rank, proving the result.    
\end{proof}

\section{Numerical Experiments}
\label{sec:exp}

We demonstrate our CUR algorithm on two datasets: (1) a document-term matrix, and (2) a gene expression dataset.  CUR has been previously applied to these types of datasets, e.g., \cite{mahoney_drineas_2009, drineas_etal_2008, sorenson_embree_2016} for document-term matrices and \cite{mahoney_drineas_2009, mairal_etal_2011} for gene expression data.  We compare performance of our CUR algorithm to that of (1) the leverage score CUR \cite{mahoney_drineas_2009} (the randomized version and deterministic variant), (2) the DEIM CUR \cite{sorenson_embree_2016}, (3) the QR CUR variant described in \cite{sorenson_embree_2016}, and (4) the low-rank SVD.  In the remainder of this paper, we denote our CUR algorithm as SF CUR (for surrogate functional CUR), and the deterministic (randomized) leverage score CUR as LS-D (LS-R) CUR.  For a brief summary of these methods, see Table \ref{tab:cur_summ}, and for more details see Section \ref{sec:cur_rel_work}.  While each CUR method selects $\mathbf{C}$ and $\mathbf{R}$ separately and allows the user to select $c$ and $r$; the LS-R CUR is not deterministic.  We include results from the LS-R CUR in comparisons of accuracy and computation time since this method is often compared to in the literature but exclude it from feature selection performance experiments.  Comparisons with the SVD are included as a baseline.

\begin{table}[h]
    \centering
    \rowcolors{2}{gray!25}{white}
    \resizebox{\textwidth}{!}{\begin{tabular}{|l|c|p{0.6\textwidth}|}
        \hline
        \textbf{Method} & \textbf{Complexity} & \textbf{Computation of $\mathbf{C}$ and $\mathbf{R}$} \\
        \hline
        LS-R CUR & $O(k'mn) + O(rmn)$ & columns (rows) sampled from a probability distribution based on leverage scores computed from the leading $k'$ right (left) singular vectors. $k'$ is a rank parameter. \\
        LS-D CUR & $O(k'mn) + O(rmn)$ & columns (rows) with largest leverage scores computed from the leading $k'$ right (left) singular vectors. $k'$ is a rank parameter. \\
        DEIM CUR & $O(kmn)$ & columns (rows) chosen using DEIM point selection algorithm on the leading $k$ right (left) singular vectors.  \\
        QR CUR & $O(mn^2)$ & columns (rows) chosen using pivoted QR of $\mathbf{X}$ ($\mathbf{C}^T$). \\
        rank-$k$ SVD & $O(kmn)$ & - \\
        \hline
    \end{tabular}}
    \caption{Summary of methods that we compare to SF CUR in this section.  The complexity given is for $\mathbf{X} \in \mathbb{R}^{m \times n}$, letting $c$ be the number of columns in $\mathbf{C}$, $r$ the number of rows of $\mathbf{R}$, and $c=r=k$.  We assume that $m < n$ and $c, r \leq m$. For each CUR method, $\mathbf{U = C^+XR^+}$.  }
    \label{tab:cur_summ}
\end{table}

Experiments were performed in MATLAB R2023b on the University of Virginia's High-Performance Computing system, Rivanna.  We used one (CPU) node, using 8 cores of an Intel(R) Xeon(R) Gold 6248 CPU at 2.50GHz, and 72GB of RAM.  Code for all experiments performed in this work is provided at \url{https://github.com/klinehan1/cur_feature_selection}.

\subsection{Document-Term Matrix}
\label{sec:dtm}

This first experiment serves to compare the accuracy and computation time of the SF CUR with that of other CUR algorithms and the SVD on a document-term matrix, $\mathbf{T} \in \mathbb{R}^{2,389 \times 21,238}$. Accuracy is given by the relative error in the Frobenius norm of each approximation, e.g., $\| \mathbf{T} -\mathbf{CUR}\|_F/\|\mathbf{T}\|_F$. $\mathbf{T}$ is sparse with $0.23\%$ non-zero entries and was downloaded and created from the 20 Newsgroups dataset \cite{data_20news} using the scikit-learn package \cite{scikit_2011}.  The documents include the training set documents for the four recreation categories; headers, footers, quotes, and a list of English stop words were removed from the text.  The documents were vectorized using TFIDF and the resulting matrix rows were normalized using the $\ell_2$ norm\footnote{The data was downloaded using the function sklearn.datasets.fetch\_20newsgroups and vectorized using sklearn.feature\_extraction.text.TfidfVectorizer. The processing was completed as described using options in these functions.  The stop word list used was the built-in list provided in scikit-learn, and the TFIDF calculations were based on the default parameter settings.}.  $\mathbf{T}$ is a rank-deficient matrix; $\text{rank}(\mathbf{T})=2,295$.    
 
Figure \ref{fig:dtm_acc} presents the relative error and Figure \ref{fig:dtm_t} presents the computation time for each CUR approximation in which $c=r$ and $c,r$ vary over $\{200, 400, ..., 2{,}200, 2{,}295\}$, and for the rank-$k$ SVD in which $k=c=r$.  Since the LS-R CUR is randomized, we ran five experiments for each value of $c,r$ and reported the average relative error and time with the standard deviations given by error bars. We note that due to sampling, the LS-R CUR may have chosen a number of columns and/or rows slightly more or less than $c$ and/or $r$. For both LS-D and LS-R CUR, the rank parameter for leverage score computation was 10. 

\begin{figure}[h]
    \centering
    \includegraphics[scale=0.7]{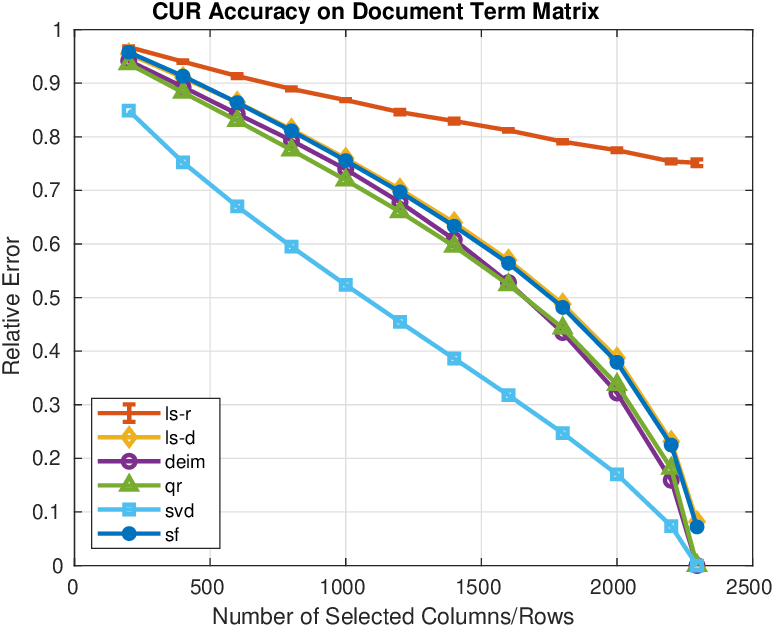}
    \caption{Relative error of CUR approximations and the rank-$k$ SVD on a document-term matrix.  The rank of the SVD approximation is the same as the number of selected columns/rows for the CUR approximations. Legend: ls-r(d), deim, qr, sf - corresponding CUR algorithms.}
    \label{fig:dtm_acc}
\end{figure}

\begin{figure}[h]
    \centering
    \includegraphics[scale=0.7]{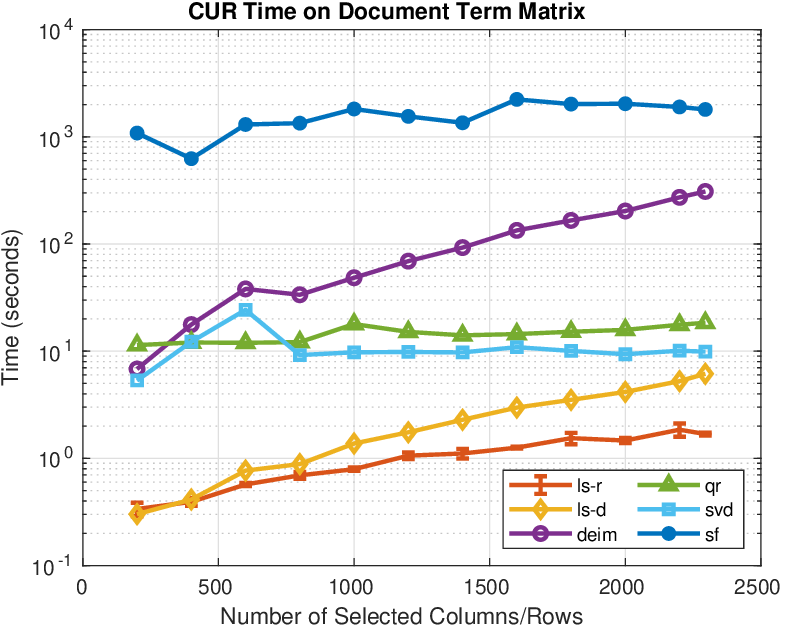}
    \caption{Computation time of CUR approximations and the rank-$k$ SVD on a document-term matrix.  The rank of the SVD approximation is the same as the number of selected columns/rows for the CUR approximations. Legend: ls-r(d), deim, qr, sf - corresponding CUR algorithms.}
    \label{fig:dtm_t}
\end{figure}

In general, the SF CUR and LS-D CUR achieve similar relative errors, as do the DEIM CUR and QR CUR, which achieve lower relative errors than those of the SF CUR and LS-D CUR.  However, for $c,r \geq 400$, the DEIM CUR has greater computation time than the QR CUR.  For $c,r \geq 1,600$, the DEIM CUR has computation times larger than 100 seconds as compared to computation times of less than 20 seconds for all values of $c,r$ for the QR CUR.  The LS-R CUR has the smallest computation time for $c,r \geq 400$, but does not perform well in relative error as $c,r$ increase.  Clearly, the SVD achieves the lowest relative error and has computation times generally about 10 seconds.  The LS-D CUR is relatively fast with computation times less than those of the SF CUR, DEIM CUR, QR CUR, and SVD.  The SF CUR is the slowest of the algorithms with computation times that are generally larger than 1,000 seconds (16.67 minutes).  While the SF CUR is the most expensive algorithm in terms of computation time, we will demonstrate its effectiveness as a feature selection method in the next experiment. 

\subsection{Gene Expression Data}
We compare the relative error and feature selection performance of the SF CUR algorithm with that of other complementary CUR algorithms and the SVD on the National Institutes of Health (NIH) gene expression dataset, GSE10072 \cite{data_gene_exp}. We repeat and extend the experiment of Sorenson and Embree \cite{sorenson_embree_2016} who compared the performance of the DEIM CUR and the LS-D CUR on this dataset by (1) calculating the error of each CUR approximation for varying values of $c,r$, and (2) assessing if the top 15 probes selected by each CUR algorithm separate the patients into those with and without a tumor. We extend this experiment by adding (1) relative error results for the SF CUR, QR CUR, and SVD, (2) computation times for each matrix approximation, (3) metrics to compare the overall probe selection of each matrix approximation method, and (4) results when selecting the top 5, 10, ..., 100 probes.  We also include relative error and computation times for the LS-R CUR on this dataset, but exclude it from the probe selection comparison since it is not deterministic.

The GSE10072 dataset, $\mathbf{G} \in \mathbb{R}^{22,283 \times 107}$, contains gene expression data for 107 patients, of which 58 have a lung tumor and 49 do not.  All entries of $\mathbf{G}$ are positive and larger entries represent a greater reaction to a probe. Each row of $\mathbf{G}$ is centered using its mean.  We approximate $\mathbf{G}^T$ (so that probes, i.e., columns, are selected first in the SF algorithm) and again use the Frobenius norm for the relative error calculation (instead of $\|\mathbf{G - CUR} \|_2$ as in \cite{sorenson_embree_2016}).  To assess how well a probe separates the patients into two classes, the number of patients in each class with a (mean-centered) entry in $\mathbf{G}^T$ greater than one for that probe are counted. As mentioned in \cite{sorenson_embree_2016}, there are 23 probes for which at least 30 patients with a tumor have an entry greater than one, and 95 probes for which at least 30 patients without a tumor have an entry greater than one. No probe is included in both of these sets. 

Figure \ref{fig:gene_acc} presents the relative error and Figure \ref{fig:gene_t} presents the computation time for each CUR and SVD approximation. Values of $c,r$ vary over $\{5,10,...,105, 107 \}$ and for each CUR approximation $c=r$. For the rank-$k$ SVD, $k=c=r$.  We report the average relative error and computation time over five runs for the LS-R CUR along with the corresponding standard deviations.  For both LS-D and LS-R CUR, the rank parameter for leverage score computation was 2.  The SF CUR and LS-D CUR have similar relative errors for all values of $c,r$; however, the SF CUR generally takes about 5 seconds to compute whereas the LS-D CUR generally takes about 0.1 seconds.  The QR CUR achieves lower relative error than the SF CUR for $c,r \geq 20$, and the DEIM CUR achieves the lowest relative error of the CUR approximations for every $c,r$ value.  The LS-R CUR has an average relative error lower than that of the SF CUR for $c,r \leq 65$. The SF CUR takes the longest to compute, while the other methods have relatively similar computation times under 0.5 seconds.  These trends are fairly similar to the relative error and computation time trends seen in the previous experiments of Section \ref{sec:dtm}.

\begin{figure}
    \centering
    \includegraphics[scale=0.7]{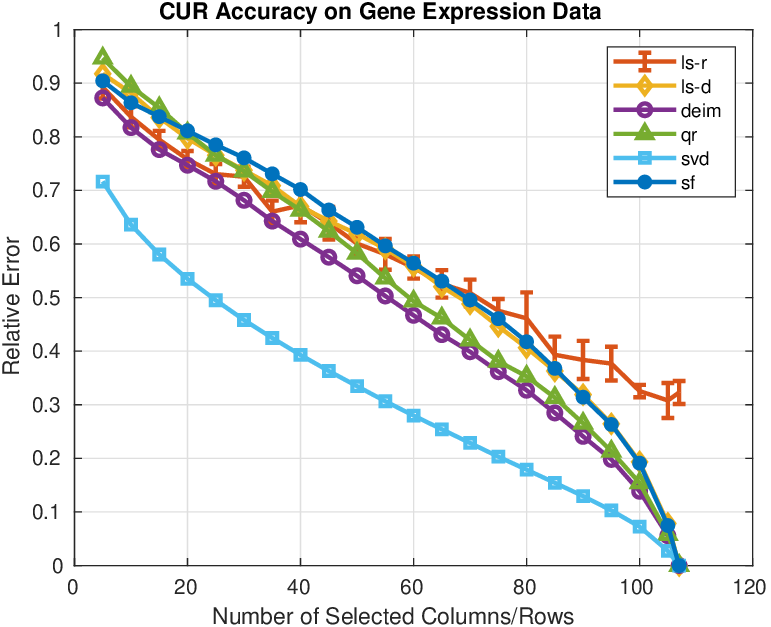}
    \caption{Relative error of CUR approximations and the rank-$k$ SVD on gene expression data.  The rank of the SVD approximation is the same as the number of selected columns/rows for the CUR approximations. Legend: ls-r(d), deim, qr, sf - corresponding CUR algorithms.}
    \label{fig:gene_acc}
\end{figure}

\begin{figure}
    \centering
    \includegraphics[scale=0.7]{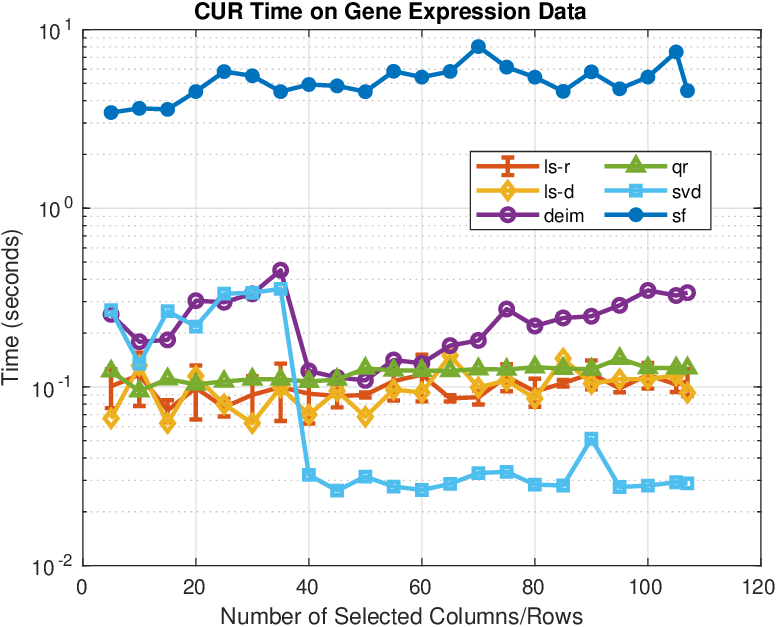}
    \caption{Time of CUR approximations and the rank-$k$ SVD on gene expression data.  The rank of the SVD approximation is the same as the number of selected columns/rows for the CUR approximations. Legend: ls-r(d), deim, qr, sf - corresponding CUR algorithms.}
    \label{fig:gene_t}
\end{figure}   

Next, we determine probe selection performance for each CUR and SVD approximation.  For each CUR approximation, we set $c$ to the corresponding number of probes, i.e., 5, 10, ..., 100, and report the selected probes (i.e., columns of $\mathbf{G}^T$).  For the rank-$k$ SVD, we perform PCA using a rank-2 SVD since the two classes (tumor and no tumor) are separated well when the data is projected onto the leading two principal axes \cite{sorenson_embree_2016}\footnote{Sorenson and Embree cite A. Kundu, S. Nambirajan, and P. Drineas [\textit{Identifying inﬂuential entries in a matrix},
preprint arXiv:1310.3556 [cs.nA], 2013] for this result; however, this paper was withdrawn from the arXiv.}, and then select the $c$ probes that have the largest correlation (in absolute value) with either the first or second principal component.  In order to compare the probe selection performance of the five methods, we compute the absolute value of the difference between the number of entries greater than one in $\mathbf{G}^T$ for patients with and without a tumor for each selected probe in each method.  See Table \ref{tab:ps_15} for an example of probe selection results using $c=15$\footnote{For this particular example, SF CUR and LS-D CUR returned the exact same set of 15 probes, hence these results are reported together in Table \ref{tab:ps_15}a.}. To quantify performance of each method, we calculate the median and mean, and standard deviation of the $c$ differences, reported in Tables \ref{tab:ps} and \ref{tab:ps_sd}, respectively. 

\begin{table}[H]
    \begin{subtable}{0.5\linewidth}
    \centering
    \small
    \rowcolors{4}{gray!25}{white}
    \begin{tabular}{|l|c|c|c|}
        \hline
        & \multicolumn{3}{c|}{\textbf{\# of Entries $>1$ in $\mathbf{G}^T$}} \\ 
        \cline{2-4}
         \textbf{Probe} & \textbf{No Tumor} & \textbf{Tumor} & $|$\textbf{Diff}$|$ \\
        \hline
        210081\_at & 45 & 2 & 43 \\
        214387\_x\_at & 48 & 6 & 42 \\
        211735\_x\_at & 48 & 5 & 43 \\
        209875\_s\_at & 2 & 50 & 48 \\
        205982\_x\_at & 48 & 5 & 43 \\
        209613\_s\_at & 47 & 2 & 45 \\
        215454\_x\_at & 46 & 0 & 46 \\
        210096\_at & 44 & 6 & 38 \\
        204712\_at & 43 & 5 & 38 \\
        203980\_at & 44 & 2 & 42 \\
        219230\_at & 38 & 2 & 36 \\
        209612\_s\_at & 46 & 2 & 44 \\
        214135\_at & 47 & 3 & 44 \\
        205866\_at & 39 & 0 & 39 \\
        205200\_at & 39 & 0 & 39 \\
        \hline
    \end{tabular}
    \caption{SF CUR and LS-D CUR.}
\end{subtable}%
    \begin{subtable}{0.5\linewidth}
    \centering
    \small
    \rowcolors{4}{gray!25}{white}
    \begin{tabular}{|l|c|c|c|}
        \hline
        & \multicolumn{3}{c|}{\textbf{\# of Entries $>1$ in $\mathbf{G}^T$}} \\ 
        \cline{2-4}
         \textbf{Probe} & \textbf{No Tumor} & \textbf{Tumor} & $|$\textbf{Diff}$|$ \\
        \hline
        210081\_at & 45 & 2 & 43 \\
        214895\_s\_at & 3 & 8 & 5 \\
        209156\_s\_at & 6 & 5 & 1 \\
        211653\_x\_at & 1 & 18 & 17 \\
        214777\_at & 3 & 27 & 24 \\
        219612\_s\_at & 17 & 17 & 0 \\
        204304\_s\_at & 4 & 16 & 12 \\
        203824\_at & 4 & 17 & 13 \\
        204748\_at & 14 & 18 & 4 \\
        201909\_at & 34 & 21 & 13 \\
        214774\_x\_at & 0 & 34 & 34 \\
        211074\_at & 4 & 7 & 3 \\
        210096\_at & 44 & 6 & 38 \\
        204475\_at & 0 & 27 & 27 \\
        214612\_x\_at & 0 & 16 & 16 \\ 
        \hline
    \end{tabular}
    \caption{DEIM CUR.}
    \end{subtable} \\ 
    \newline \quad \newline
    \begin{subtable}{0.5\linewidth}
    \centering
    \small
    \rowcolors{4}{gray!25}{white}
    \begin{tabular}{|l|c|c|c|}
        \hline
        & \multicolumn{3}{c|}{\textbf{\# of Entries $>1$ in $\mathbf{G}^T$}} \\ 
        \cline{2-4}
         \textbf{Probe} & \textbf{No Tumor} & \textbf{Tumor} & $|$\textbf{Diff}$|$ \\
        \hline
        214387\_x\_at & 48 & 6 & 42 \\
        201909\_at & 34 & 21 & 13 \\
        206239\_s\_at & 2 & 30 & 28 \\
        205725\_at & 33 & 7 & 26 \\
        219612\_s\_at & 17 & 17 & 0 \\
        203290\_at & 29 & 19 & 10 \\
        213831\_at & 28 & 18 & 10\\
        213674\_x\_at & 8 & 17 & 9 \\
        204475\_at & 0 & 27 & 27 \\
        201884\_at & 0 & 30 & 30 \\
        209278\_s\_at & 0 & 21 & 21 \\
        204885\_s\_at & 12 & 18 & 6 \\
        207430\_s\_at & 6 & 9 & 3 \\
        205476\_at & 11 & 16 & 5 \\
        209309\_at & 7 & 17 & 10 \\
        \hline
    \end{tabular}
    \caption{QR CUR.}
    \end{subtable}%
        \begin{subtable}{0.5\linewidth}
    \centering
    \small
    \rowcolors{4}{gray!25}{white}
    \begin{tabular}{|l|c|c|c|}
        \hline
        & \multicolumn{3}{c|}{\textbf{\# of Entries $>1$ in $\mathbf{G}^T$}} \\ 
        \cline{2-4}
         \textbf{Probe} & \textbf{No Tumor} & \textbf{Tumor} & $|$\textbf{Diff}$|$ \\
        \hline
        204931\_at & 36 & 0 & 36 \\
        206702\_at & 34 & 0 & 34 \\
        201540\_at & 43 & 0 & 43 \\
        209074\_s\_at & 45 & 1 & 44 \\
        206742\_at & 38 & 0 & 38 \\
        205200\_at & 39 & 0 & 39 \\
        219213\_at & 19 & 0 & 19 \\
        204894\_s\_at & 29 & 0 & 29 \\
        213103\_at & 5 & 0 & 5 \\
        206209\_s\_at & 36 & 0 & 36 \\
        219719\_at & 13 & 0 & 13 \\
        204719\_at & 42 & 0 & 42 \\
        208981\_at & 22 & 0 & 22 \\
        210081\_at & 45 & 2 & 43 \\
        204396\_s\_at & 36 & 0 & 36 \\
        \hline
    \end{tabular}
    \caption{SVD.}
    \end{subtable} 
    \caption{Probe selection results for $c=15$. Probes are listed in ranked order (1 to 15) for LS-D CUR, DEIM CUR, QR CUR and SVD methods since these methods return selected columns in a ranked order.}
    \label{tab:ps_15}
\end{table}

\begin{table}[h]
    \small
    \rowcolors{8}{gray!25}{white}
    \resizebox{\textwidth}{!}{\begin{tabular}{|c|cc|cc|cc|cc|cc|}
        \hline
        & \multicolumn{2}{c|}{\textbf{SF}} & \multicolumn{2}{c|}{\textbf{LS-D}} & \multicolumn{2}{c|}{\textbf{DEIM}} & \multicolumn{2}{c|}{\textbf{QR}} & 
        \multicolumn{2}{c|}{\textbf{SVD}} \\
        \cline{2-3} \cline{4-5} \cline{6-7} \cline{8-9} \cline{10-11}%
        \textbf{\# of Probes} & \textbf{mdn} & \textbf{mean} & \textbf{mdn} & \textbf{mean} & \textbf{mdn} & \textbf{mean} & \textbf{mdn} & \textbf{mean} & \textbf{mdn} & \textbf{mean} \\
        \hline
5 & \textbf{43} & \textbf{43.80} & \textbf{43} & \textbf{43.80} & 17 & 18 & 26 & 21.80 & 38 & 39 \\
10 & \textbf{43} & \textbf{42.80} & \textbf{43} & \textbf{42.80} & 12.50 & 13.20 & 19.50 & 19.50 & 36 & 32.30 \\
15 & \textbf{43} & \textbf{42} & \textbf{43} & \textbf{42} & 13 & 16.67 & 10 & 16 & 36 & 31.93 \\
20 & \textbf{42.50} & 41.75 & \textbf{42.50} & \textbf{42.10} & 13 & 15.80 & 10 & 14.55 & 36 & 31.95 \\
25 & \textbf{42} & \textbf{40.68} & \textbf{42} & 40.32 & 14 & 16.28 & 13 & 15.56 & 34 & 31.32 \\
30 & \textbf{40} & 39.20 & \textbf{40} & \textbf{39.30} & 13.50 & 15.57 & 13 & 15.87 & 34 & 30.67 \\
35 & \textbf{39} & 38.66 & \textbf{39} & \textbf{38.69} & 14 & 15.86 & 13 & 16.40 & 31 & 29.83 \\
40 & \textbf{38} & 38.28 & \textbf{38} & \textbf{38.30} & 14.50 & 16.15 & 16.50 & 17.25 & 30 & 29.43 \\
45 & \textbf{38} & \textbf{38.16} & \textbf{38} & 38.11 & 16 & 16.89 & 17 & 17.73 & 31 & 30.02 \\
50 & \textbf{37} & \textbf{37.94} & 36.50 & 37.76 & 15.50 & 15.92 & 16.50 & 16.96 & 30 & 29.36 \\
55 & \textbf{36} & \textbf{37.42} & \textbf{36} & 37.31 & 14 & 15.02 & 16 & 16.62 & 29 & 28.49 \\
60 & \textbf{36} & 37.12 & \textbf{36} & \textbf{36.78} & 14.50 & 14.93 & 14.50 & 16.10 & 29 & 27.63 \\
65 & \textbf{36} & \textbf{36.69} & \textbf{36} & 36.52 & 14 & 14.65 & 16 & 16 & 29 & 27.55 \\
70 & \textbf{36} & \textbf{36.31} & \textbf{36} & 36.17 & 15.50 & 15.01 & 16 & 15.89 & 28 & 27 \\
75 & \textbf{36} & \textbf{36.13} & \textbf{36} & 35.95 & 14 & 15.05 & 16 & 16.05 & 28 & 26.61 \\
80 & \textbf{35.50} & \textbf{35.84} & 35 & 35.64 & 13 & 14.49 & 15.50 & 16.03 & 27.50 & 26.08 \\
85 & \textbf{35} & \textbf{35.54} & \textbf{35} & 35.44 & 13 & 14.25 & 15 & 16.07 & 28 & 26.04 \\
90 & \textbf{35} & \textbf{35.20} & \textbf{35} & 35.14 & 13 & 14.12 & 14 & 15.97 & 28 & 25.89 \\
95 & \textbf{35} & \textbf{34.95} & \textbf{35} & 34.82 & 13 & 13.66 & 14 & 16.06 & 28 & 25.88 \\
100 & \textbf{34} & \textbf{34.63} & \textbf{34} & 34.49 & 12.50 & 13.46 & 14 & 16.06 & 26.50 & 25.62 \\
        \hline
    \end{tabular}}
    \caption{Comparison of probe selection methods using the metrics of median (mdn) and mean difference between classes.  A greater difference implies better performance, i.e., better separation of the tumor and no tumor classes by the selected probes. The greatest median and mean difference are bolded in each row.}
    \label{tab:ps}
\end{table}

\begin{table}[h]
    \centering
    \small
    \rowcolors{2}{gray!25}{white}
    \begin{tabular}{|c|c|c|c|c|c|}
        \hline
         \textbf{\# of Probes} & \textbf{SF} & \textbf{LS-D} & \textbf{DEIM} & \textbf{QR} & \textbf{SVD} \\
        \hline
5 & \textbf{2.39} & \textbf{2.39} & 16.73 & 15.94 & 4.36 \\
10 & \textbf{3.16} & \textbf{3.16} & 12.89 & 12.91 & 11.98 \\
15 & \textbf{3.36} & \textbf{3.36} & 13.80 & 12.17 & 11.93 \\
20 & 3.34 & \textbf{2.94} & 12.49 & 11.20 & 10.68 \\
25 & \textbf{3.94} & 4.80 & 11.62 & 10.64 & 9.76 \\
30 & 5.07 & \textbf{4.97} & 11.17 & 10.39 & 10.65 \\
35 & \textbf{5.21} & 5.23 & 10.75 & 11.14 & 10.40 \\
40 & \textbf{4.99} & 5.02 & 10.87 & 11.26 & 10 \\
45 & \textbf{4.88} & 4.90 & 10.54 & 11.13 & 10.22 \\
50 & \textbf{4.68} & 4.86 & 10.58 & 11.24 & 10.44 \\
55 & \textbf{4.86} & 4.97 & 10.62 & 11.27 & 10.37 \\
60 & \textbf{4.89} & 5.11 & 10.30 & 10.96 & 10.89 \\
65 & \textbf{5.02} & 5.14 & 10.15 & 11.02 & 10.54 \\
70 & \textbf{5.05} & 5.22 & 10.35 & 10.74 & 10.36 \\
75 & \textbf{5.04} & 5.16 & 10.67 & 11.03 & 10.33 \\
80 & \textbf{5.02} & 5.16 & 10.63 & 11.05 & 11.10 \\
85 & \textbf{5.06} & 5.19 & 10.46 & 11.43 & 11.30 \\
90 & \textbf{5.13} & 5.19 & 10.32 & 11.46 & 11.32 \\
95 & \textbf{5.11} & 5.30 & 10.24 & 11.31 & 11.26 \\
100 & \textbf{5.23} & 5.39 & 10.10 & 11.24 & 11.42 \\
\hline
    \end{tabular}
    \caption{Standard deviation results for probe selection methods.  The smallest standard deviation of differences is bolded in each row.} 
    \label{tab:ps_sd}
\end{table}

The probes selected by SF CUR and LS-D CUR perform very well in separating patients with a tumor from those without a tumor as they have larger median and mean differences and smaller standard deviations of differences than the probes selected by the other methods.  The probes selected by SVD also perform fairly well in this task due to their fairly large median and mean differences, but they exhibit standard deviations that are generally double those of the SF CUR and LS-D CUR probes.  The probes selected by DEIM CUR and QR CUR perform poorly in median and mean difference and exhibit standard deviations that are generally double those of the SF CUR and LS-D CUR probes as well.  

While the SF CUR and LS-D CUR methods achieve very similar results\footnote{These two methods select many of the same probes for each value of $c$.  See Table S1 in the supplementary material.}, the SF CUR outperforms the LS-D CUR in this experiment.  For 3/20 values of $c$, the probes selected by the two methods produce equal values for the median and mean difference, and standard deviation of differences (since they select the same exact probes). However, the SF CUR probes achieve a median difference greater than or equal to that of the LS-D CUR probes for all values of $c$, and a mean difference greater than or equal to that of the LS-D CUR probes for 12/20 values of $c$.  Additionally, the standard deviation of differences for the SF CUR probes is less than that of the LS-D CUR probes for 15/20 values of $c$.  
 
\section{Protein Expression Discriminant Analysis with CUR}
\label{sec:pro_exp}

Lastly, we present a novel application of CUR as a feature selection method for a clustering algorithm on protein expression data.  We modify the experiments of Higuera et al. \cite{higuera_et_al_2015} in which discriminant proteins that critically affect learning in wild type and trisomic mice were discovered in biologically relevant pairwise class comparisons with clustering provided by SOMs and feature selection provided by the Wilcoxon rank-sum test.  Specifically, we demonstrate the use and effectiveness of CUR as the feature selection method in a subset of these computational experiments.  We compare the performance of multiple CUR algorithms in this application.

\subsection{Prior Computational Experiments}
In this section we provide a summary of the dataset used and computational experiments (MATLAB R2011b) performed in \cite{higuera_et_al_2015}.  

\subsubsection{Data}
The protein expression data used in these experiments was measured from two groups of mice, control and trisomic.  Each mouse was exposed to one option from each of two treatments:
\begin{enumerate}
    \item Context fear conditioning (CFC), an associative learning assessment task, of either context-shock (CS) or shock-context (SC), and
    \item an injection of memantine, a drug known to treat learning impairment, or saline. 
\end{enumerate}
The CFC task consisted of placing a mouse in a novel cage and allowing it to explore. The context-shock option involves giving the mouse an electric shock after a few minutes of cage exploration, whereas the shock-context option involves an immediate shock to the mouse before exploration.  Control mice given the context-shock option will learn an association between cage and shock, whereas those given the shock-context option will not.  However, trisomic mice given the context-shock option will not learn the association between cage and shock. Thus, the second treatment, an injection of memantine or saline, is given prior to the CFC task.  Trisomic mice injected with memantine prior to the CFC context-shock task will learn the association as the control mice do.  Learning in control mice is not affected by memantine injection.  Table \ref{tab:mice} summarizes the eight classes of mice and presents class size and type of learning.  For the remainder of this work, we will refer to classes by their names in Table \ref{tab:mice}.

\begin{table}
    \centering
    \rowcolors{2}{gray!25}{white}
    \resizebox{\textwidth}{!}{\begin{tabular}{|l|l|l|l|l|c|}
       \hline 
       \textbf{Class Name} & \textbf{Genotype} & \textbf{CFC (Stimulation to Learn)} & \textbf{Injection}  & \textbf{Learning} & \textbf{Class Size} \\
       \hline
       c-CS-s & control & context-shock (yes) & saline & Normal & 9 \\
       c-CS-m & control & context-shock (yes) & memantine & Normal & 10 \\ 
       c-SC-s & control & shock-context (no) & saline & None & 9 \\
       c-SC-m & control & shock-context (no) & memantine & None & 10 \\
       t-CS-s & trisomic & context-shock (yes) & saline & Failed & 7 \\
       t-CS-m & trisomic & context-shock (yes) & memantine & Rescued & 9 \\ 
       t-SC-s & trisomic & shock-context (no) & saline & None & 9 \\
       t-SC-m & trisomic & shock-context (no) & memantine & None & 9 \\
       \hline
    \end{tabular}}
    \caption{Classes of 72 total mice.}
    \label{tab:mice}
\end{table}

Data consist of expression levels for 77 proteins measured from the brains of the 72 mice represented in Table \ref{tab:mice}.  Each protein was measured 15 times for each mouse giving a total of 1,080 measurements of 77 proteins, resulting in a data matrix that is 1,080 $\times$ 77. For each of the 1,080 observations, the mouse id and class of the mouse that produced the measurements are also provided.  The dataset is available in the supporting information of \cite{higuera_et_al_2015} and in the University of California, Irvine Machine Learning Repository \cite{data_pro_exp}.  

Missing data arises as a consequence of the protein measurement process. Higuera et al. processed the data by (1) removing an outlier mouse with mainly missing data, (2) filling in missing entries, and (3) normalizing the data.  For any mice in class $c$ missing data for protein $p$, the missing entries were replaced with the average expression level of protein $p$ from the reported (non-missing) entries for mice in class $c$.  Min-max normalization was then applied to each column of the data, i.e., for each protein.  When we analyzed the raw data, we could not identify the outlier mouse and suspect the data for this mouse was already removed from the dataset before it was posted for download. We found that 1.7\% of entries in the data were missing with only 6 out of 77 proteins missing 20 or more measurements out of the 1,080 total. We also discovered that two columns of the raw data are equal; these columns correspond to the proteins ARC\_N and pS6\_N.  These columns are clearly the same after the data is processed as well, and thus make the matrix of protein expression data rank deficient. 

\subsubsection{Methodology}

We first give a high-level overview of the methodology and then follow with more details.  SOMs and the Wilcoxon rank-sum test\footnote{also called the Mann-Whitney U test.} were used to discover discriminant proteins in biologically relevant pairwise class comparisons. An SOM is an unsupervised neural network clustering method that can identify the topology and distribution of data such that clusters that exist close together in the topology should cluster similar data points.  It is useful for dimension reduction and provides 2D visualization of the data.  An SOM is used in this case to cluster mice with similar protein expression levels in order to discover protein expression patterns among classes.  The data provided to the SOM include the protein expression data, but not the class of each mouse.  The Wilcoxon rank-sum test is then used to find discriminant proteins between pairs of SOM ``class-specific clusters'' of mice (details below).  This nonparametric test checks for equal medians between two independent samples of data that are not necessarily the same size. 

Higuera et al. used this method on (1) data from the four control classes, (2) data from the four trisomic classes, and (3) a mixture of the two (c-CS-s, c-CS-m, t-CS-m, t-CS-s, and c-SC-s).  We explore feature selection with CUR instead of the Wilcoxon rank-sum test on data from the four control classes only.  The particular experiments are detailed below.

Initially, ten $7 \times 7$ SOMs are computed on the processed data from the four control classes, a 570 $\times$ 77 dataset. Since SOM neuron weights are initialized randomly, each SOM instance will most likely be different. The average quantization error, $q$, is measured for each SOM as  
\[
q = \frac{1}{n} \sum_{i=1}^n \|\mathbf{d}_i - \mathbf{w}_{BMU(\mathbf{d}_i)}\|_2, 
\]
where $n$ is the number of observations, $\mathbf{d}_i \in \mathbb{R}^{77}$ is an observation (a row of the data matrix), and $\mathbf{w}_{BMU(\mathbf{d}_i)} \in \mathbb{R}^{77}$ is the weight vector of the Best Matching Unit (BMU) or closest neuron to $\mathbf{d}_i$.  The SOM with the smallest average quantization error is then used to identify ``class-specific clusters'', defined in \cite{higuera_et_al_2015} as ``(i) two or more adjacent [neurons] that contain mice of the same class and no mice from other classes, or (ii) a single [neuron] that contains $\geq$ 80\% (or $\geq$ 12 of 15) of the measurements of one mouse and no measurements of mice from any other class.'' While the class of each mouse was not used in the learning process of the SOM, the class of each mouse is used to determine the class-specific clusters.  Two class-specific clusters can be compared using the weight vectors of the neurons included in the cluster and 77 instances of the Wilcoxon rank-sum test, one for each protein (each neuron weight vector is length 77).  For example, to compare levels of the protein in column 5 of the dataset, the Wilcoxon rank-sum test would use two samples: one created from the fifth element of each neuron weight vector for the neurons in the first class-specific cluster, and one created from the fifth element of each neuron weight vector for the neurons in the second class-specific cluster.  Those proteins for which the Wilcoxon rank-sum test returns a $p$-value of less than 0.05 are considered to be the discriminant proteins between the two class-specific clusters and thus the two classes they represent. 

A new $7 \times 7$ SOM is then created using data for the discriminant proteins only (a 570 $\times$ $k$ matrix where $k$ is the number of discriminant proteins). Class-specific clusters are identified in this SOM. The discriminant proteins are validated through a qualitative analysis of this SOM, which includes the number of mixed class neurons and the number of observations in mixed class neurons as metrics to determine how well the discriminant proteins clustered the data.   

In particular, Higuera et al. found the common discriminant proteins to four pairwise comparisons involving successful learning, c-CS-s vs. c-SC-s, c-CS-m vs. c-SC-m, c-CS-m vs. c-SC-s, and c-CS-s vs. c-SC-m and then used these results in two other experiments.
\begin{enumerate}
    \item Experiment 1 discriminant proteins: The union of those between c-CS-m and c-CS-s and the common discriminant proteins between the four successful learning comparisons.
    \item Experiment 2 discriminant proteins: The union of those between c-SC-m and c-SC-s and the common discriminant proteins between the four successful learning comparisons.
\end{enumerate}
Each experiment produced an SOM that was qualitatively analyzed in order to validate the selection of discriminant proteins.  In the next two subsections of this work we describe how we used CUR as a feature selection method in this methodology and provide results for its use in Experiments 1 and 2. 

\subsection{Feature Selection Using CUR}

To use CUR as a feature selection method between two class-specific clusters, we construct a matrix $\mathbf{D}$ that contains pairwise differences between neuron weight vectors in opposite clusters.  For example if cluster A contains $a$ neurons (call their weight vectors $\mathbf{A}_1, ..., \mathbf{A}_a$) and cluster B contains $b$ neurons (call their weight vectors $\mathbf{B}_1, ..., \mathbf{B}_{b}$), then $\mathbf{D} \in \mathbb{R}^{ab \times 77}$, and 
\[
\mathbf{D}(j+b(i-1),:) = \mathbf{A}_{i} - \mathbf{B}_{j},
\]
for $i \in [a]$ and $j \in [b]$.  We then compute 77 CUR approximations of $\mathbf{D}$; one for each possible number of columns to select for the matrix $\mathbf{C}$, i.e., $1,2,...,77$\footnote{Since two columns of the raw protein expression data are equal, $\mathbf{D}$ also has this property.  Thus, the SF CUR may fail to produce a selection of columns for particular values of the number of columns to select.  In these cases, we ignore these values of the number of columns to select.}.  For this particular application, we are only interested in the chosen subset of columns selected for $\mathbf{C}$, and in each CUR approximation that we use, the columns are chosen first, independently of the rows.  Hence, we could select any number of rows for the matrix $\mathbf{R}$. We chose to use all rows and set $\mathbf{R} = \mathbf{D}$.  In order to select a CUR approximation from the 77 calculated, we compute the Akaike information criterion (AIC) \cite{akaike_1973} and the Bayesian information criterion (BIC) \cite{schwarz_1978} for each CUR model as given in the formulas below.  Let $\mathbf{D} \in \mathbb{R}^{m \times 77}$, and $\mathbf{CUR}$ be the CUR approximation to $\mathbf{D}$ where $\mathbf{C} \in \mathbb{R}^{m \times c}$.  Then,
\[
\text{AIC} = 2(mc+3) + 77m \ln \left( \frac{\| \mathbf{D} - \mathbf{CUR} \|_F^2}{77m} \right), 
\]
and
\[
\text{BIC} = (mc+3) \ln (77m) + 77m \ln \left( \frac{\| \mathbf{D} - \mathbf{CUR} \|_F^2}{77m} \right). 
\]
The CUR model with the lowest AIC score and the CUR model with the lowest BIC score are selected.  The columns chosen to be in the matrix $\mathbf{C}$ of each CUR correspond to the discriminant proteins.  We then train two SOMs: the first using the discriminant proteins from the CUR model with the lowest AIC and the second using the discriminant proteins from the CUR model with the lowest BIC.  Hence, using CUR for feature selection will result in two possibilities for the set of discriminant proteins, which can be compared through a qualitative assessment of the SOMs trained on each set.  

\subsection{Results}

We repeated Experiments 1 and 2 by Higuera et al. with two exceptions: (1) we used CUR as the feature selection process instead of the Wilcoxon rank-sum test, and (2) each time we needed to use an SOM, we trained 10 SOMs and chose the one with the minimum number of mixed-class neurons. If multiple SOMs had the minimum number of mixed-class neurons, we chose the SOM with the minimum number of observations in mixed-class neurons. We focused on these metrics based on their importance in the qualitative analysis of the discriminant protein SOMs in \cite{higuera_et_al_2015}.  We compared the performance of four CUR algorithms in this feature selection task, SF CUR, LS-D CUR, DEIM CUR, and QR CUR, as well as that of the Wilcoxon rank-sum test.  In all experiments, the LS-D CUR rank parameter for leverage score computation was 2. All experiments were run in MATLAB R2023b on a laptop with 16GB RAM and a 2.20 GHz Intel Core i7-1360P processor\footnote{The SOM implementation in MATLAB's Deep Learning Toolbox was used with the default parameters except the SOM size.}.

Figure \ref{fig:som_all} presents the SOM using all 77 proteins.  Neurons are labeled with the classes of mice they contain in sorted order, i.e., the first class listed is the majority class, and the number of observations contained in the neuron.  Neurons are colored by their majority class -  c-CS-m: yellow, c-CS-s: green, c-SC-m: tan, c-SC-s: orange - and a bold outline of a neuron represents class-specific cluster membership - c-CS-m class cluster: red outline, c-CS-s class cluster: green outline, c-SC-m class cluster: brown outline, c-SC-s class cluster: black outline.  This color scheme is based on that in \cite{higuera_et_al_2015}. 

\begin{figure}
    \centering
    \includegraphics{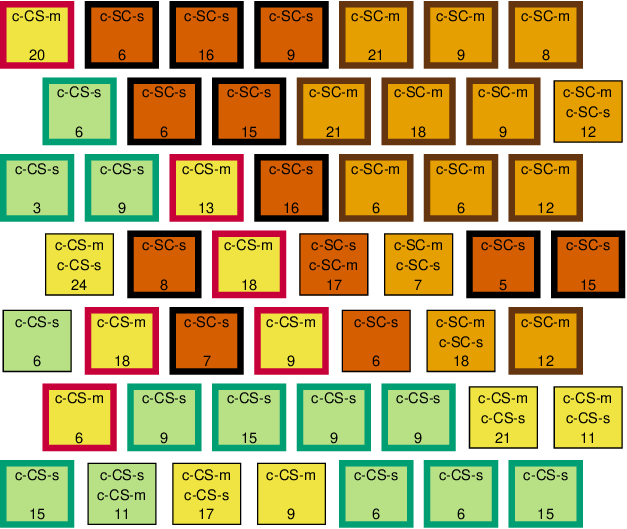}
    \caption{SOM using all 77 proteins.}
    \label{fig:som_all}
\end{figure}

We define a mixed-CS-class neuron as a mixed-class neuron that includes either c-CS-m or c-CS-s observations, and a mixed-SC-class neuron as a mixed-class neuron that includes either c-SC-m or c-SC-s observations. As a reference for comparison in the results of Experiments 1 and 2, the SOM in Figure \ref{fig:som_all} has five mixed-CS-class neurons which contain 84 observations and four mixed-SC-class neurons which contain 54 observations.  The goal of Experiment 1 is to select discriminant proteins such that when an SOM is trained on the discriminant protein data only, the SOM improves, i.e., has a smaller number of mixed-CS-class neurons and observations contained within those neurons, as compared to the SOM in Figure \ref{fig:som_all} that was trained on all protein data.  The goal of Experiment 2 is similar except that the discriminant protein SOM should have a smaller number of mixed-SC-class neurons and observations contained within those neurons.  Since each CUR algorithm results in two potential sets of discriminant proteins (one for the CUR with minimum AIC and one for the CUR with minimum BIC), we present results for nine feature selection methods: (1) Wilcoxon rank-sum test, (2-3) SF CUR AIC/BIC, (4-5) LS-D CUR AIC/BIC, (6-7) DEIM CUR AIC/BIC, and (8-9) QR CUR AIC/BIC.  

\subsubsection{Experiment 1}

For each feature selection method we (1) identified the common discriminant proteins between the four successful learning comparisons, c-CS-s vs. c-SC-s, c-CS-m vs. c-SC-m, c-CS-m vs. c-SC-s, and c-CS-s vs. c-SC-m, and (2) identified the discriminant proteins between the c-CS-m and c-CS-s classes. The union of these two sets of proteins is the set of discriminant proteins.  We present the results of Experiment 1 in Table \ref{tab:CS_res}. In addition, Figure \ref{fig:CS_aic} contains the discriminant protein SOMs for the Wilcoxon rank-sum test and each CUR algorithm using the AIC model selection criteria.  Since the CUR algorithms using the AIC criteria generally outperformed those using the BIC criteria, the discriminant protein SOMs for the CUR algorithms using the BIC model selection criteria are found in Figure 1 of the supplementary material. 

\begin{figure}[H]
    \centering
    \begin{subfigure}{.5\textwidth}
        \centering
        \includegraphics[width=0.9\textwidth]{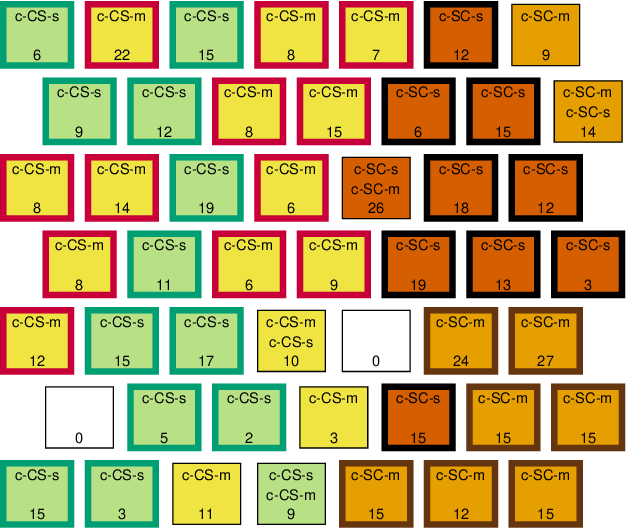}
        \caption{Wilcoxon rank-sum test}
    \end{subfigure}%
    \begin{subfigure}{.5\textwidth}
        \centering
        \includegraphics[width=0.9\textwidth]{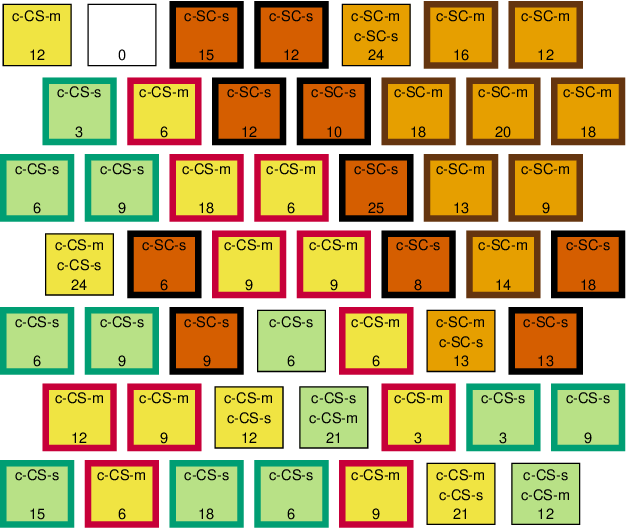}
        \caption{SF CUR, AIC}
    \end{subfigure}
    \par \bigskip \smallskip 
    \begin{subfigure}{.5\textwidth}
        \centering
        \includegraphics[width=0.9\textwidth]{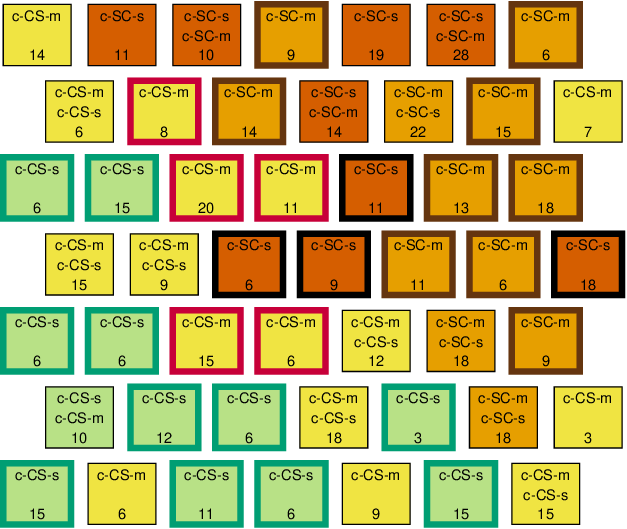}
        \caption{LS-D CUR, AIC}
    \end{subfigure}%
    \begin{subfigure}{.5\textwidth}
        \centering
        \includegraphics[width=0.9\textwidth]{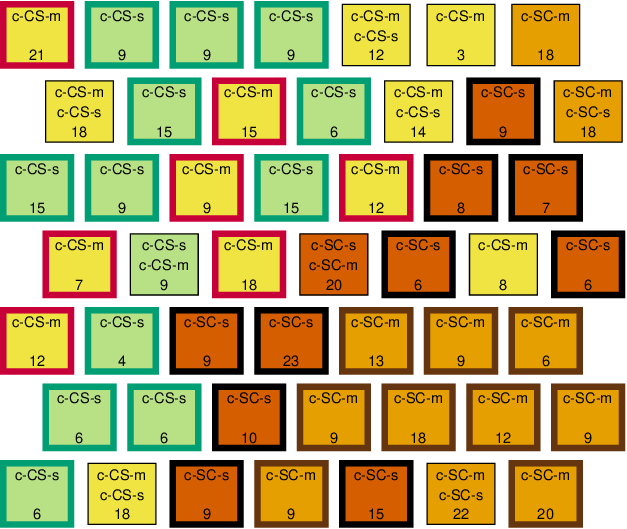}
         \caption{DEIM CUR, AIC}
    \end{subfigure}
    \par \bigskip \smallskip 
    \begin{subfigure}{.5\textwidth}
        \centering
        \includegraphics[width=0.9\textwidth]{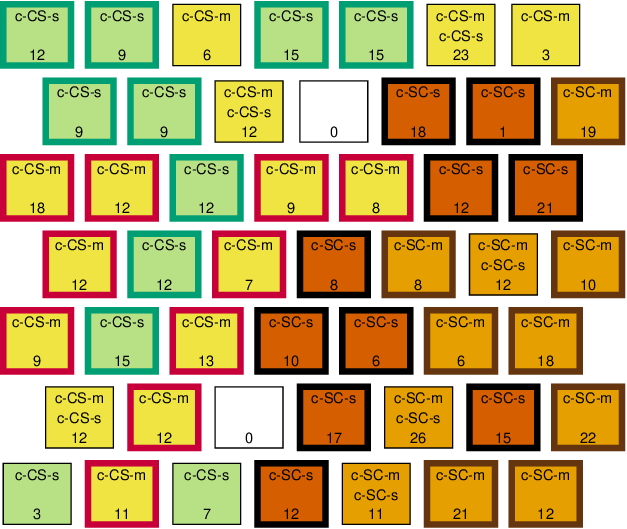}
         \caption{QR CUR, AIC}
    \end{subfigure}
    \caption{Experiment 1 discriminant protein SOMs for the Wilcoxon rank-sum test and CUR algorithms using the AIC model selection criteria.}
    \label{fig:CS_aic}
\end{figure}

\begin{table}[h]
    \centering
    \rowcolors{2}{gray!25}{white}
    \begin{tabular}{|c|c|c|c|}
        \hline
        \textbf{Feature Selection} & \textbf{\# Mixed-CS} & \textbf{\# Observations in} & \textbf{\# Discriminant} \\
        \textbf{Method} & \textbf{Neurons} & \textbf{Mixed-CS Neurons} & \textbf{Proteins} \\
        \hline
        None & 5 & 84 & - \\
        Wilcoxon rank-sum test & \textbf{2} & \textbf{19} & 15 \\
        SF CUR - AIC  &  5 & 90 & 35 \\
        SF CUR - BIC  & 7 & 91 & 21 \\
        LS-D CUR - AIC & 7 & 85 & 49 \\
        LS-D CUR - BIC & 6 &  105 & 28 \\
        DEIM CUR - AIC  & 5 & 71 & 20 \\
        DEIM CUR - BIC  & 7 & 90 & 18 \\
        QR CUR - AIC  & 3 & 47 & 18 \\
        QR CUR - BIC & 7 & 77 & 16 \\
        \hline
    \end{tabular}
    \caption{Experiment 1 results.  The minimum numbers of mixed-CS-class neurons and observations are in bold.}
    \label{tab:CS_res}
\end{table}

The Wilcoxon rank-sum test performs the best of the feature selection methods, resulting in two mixed-CS-class neurons and 19 observations in those neurons, which is by far the minimum number of observations in mixed-CS-class neurons.  It is interesting to note that not only does the Wilcoxon rank-sum test perform the best, but it also selects the fewest number of discriminant proteins.  Amongst CUR algorithms, the QR CUR - AIC performs the best, resulting in only three mixed-CS-class neurons containing 47 observations.  All CUR algorithms except the QR CUR - AIC, QR CUR - BIC, and DEIM CUR - AIC perform worse than the baseline of no feature selection. Although, the QR CUR - BIC results are mixed: there are two more mixed-CS-class neurons, but these neurons contain less observations than that of the baseline.  While CUR-based feature selection did not perform as well as the Wilcoxon rank-sum test for this experiment, we will see a different result in Experiment 2.  

\subsubsection{Experiment 2}

For each feature selection method we (1) again identified the common discriminant proteins between the four successful learning comparisons, c-CS-s vs. c-SC-s, c-CS-m vs. c-SC-m, c-CS-m vs. c-SC-s, and c-CS-s vs. c-SC-m, and (2) identified the discriminant proteins between the c-SC-m and c-SC-s classes. The union of these two sets of proteins is the set of discriminant proteins.  Results for Experiment 2 are presented in Table \ref{tab:SC_res}.  The discriminant protein SOMs for the Wilcoxon rank-sum test and CUR algorithms using the AIC model selection criteria are given in Figure \ref{fig:SC_aic}.  The discriminant protein SOMs for CUR algorithms using the BIC model selection criteria are given in Figure 2 of the supplementary material. 

\begin{figure}[H]
    \centering
    \begin{subfigure}{.5\textwidth}
        \centering
        \includegraphics[width=0.9\textwidth]{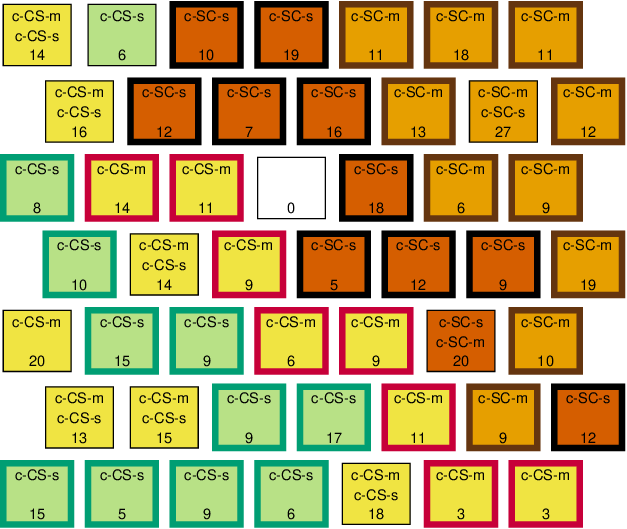}
        \caption{Wilcoxon rank-sum test}
    \end{subfigure}%
    \begin{subfigure}{.5\textwidth}
        \centering
        \includegraphics[width=0.9\textwidth]{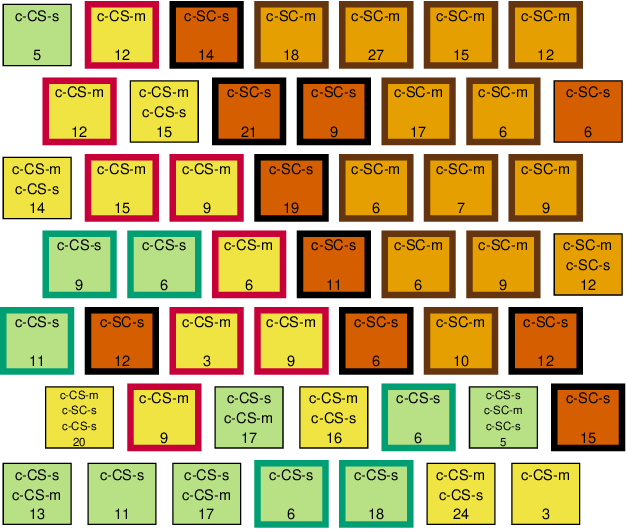}
        \caption{SF CUR, AIC}
    \end{subfigure}
    \par \bigskip \smallskip 
    \begin{subfigure}{.5\textwidth}
        \centering
        \includegraphics[width=0.9\textwidth]{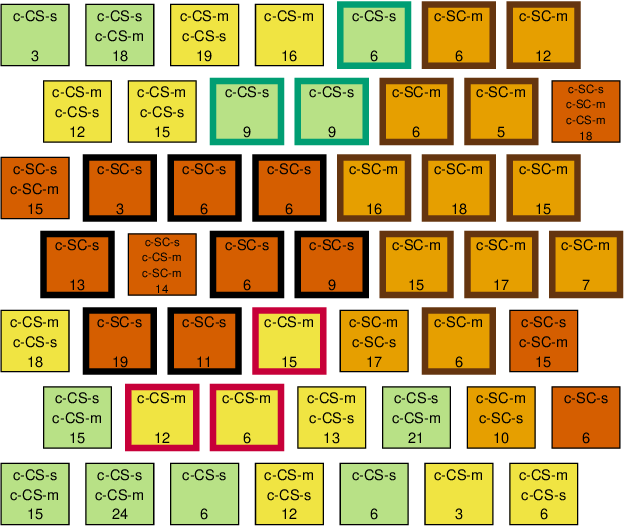}
        \caption{LS-D CUR, AIC}
    \end{subfigure}%
    \begin{subfigure}{.5\textwidth}
        \centering
        \includegraphics[width=0.9\textwidth]{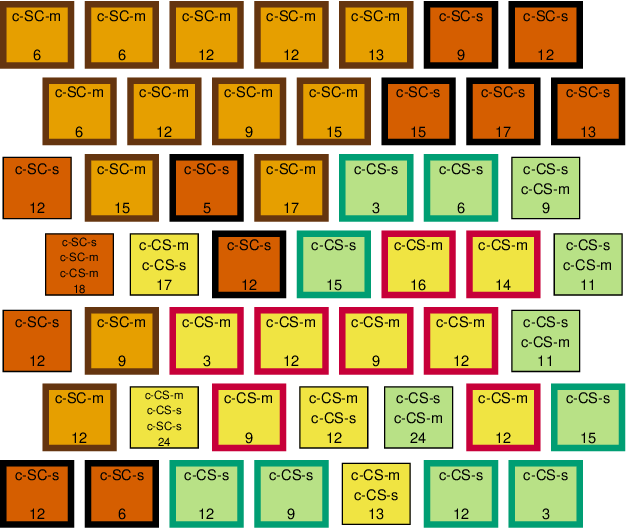}
         \caption{DEIM CUR, AIC}
    \end{subfigure}
    \par \bigskip \smallskip 
    \begin{subfigure}{.5\textwidth}
        \centering
        \includegraphics[width=0.9\textwidth]{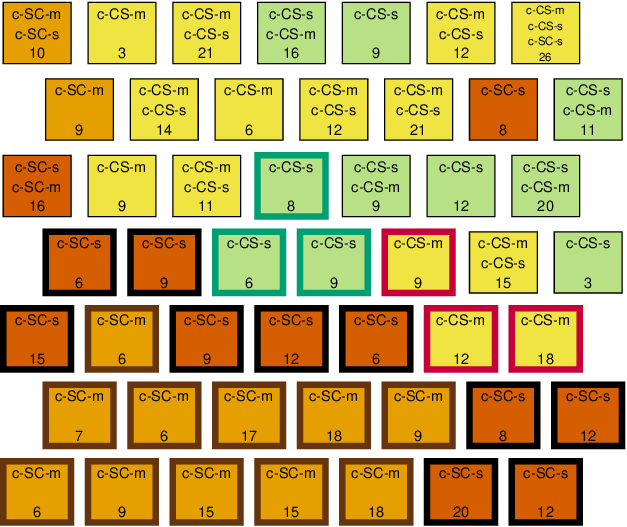}
         \caption{QR CUR, AIC}
    \end{subfigure}
    \caption{Experiment 2 discriminant protein SOMs for the Wilcoxon rank-sum test and CUR algorithms using the AIC model selection criteria.}
    \label{fig:SC_aic}
\end{figure}

\begin{table}[h]
    \centering
    \rowcolors{3}{gray!25}{white}
   \begin{tabular}{|c|c|c|c|}
        \hline
        \textbf{Feature Selection} & \textbf{\# Mixed-SC} & \textbf{\# Observations in} & \textbf{\# Discriminant} \\
        \textbf{Method} & \textbf{Neurons} & \textbf{Mixed-SC Neurons} & \textbf{Proteins} \\
        \hline
        None & 4 & 54 & - \\
        Wilcoxon Rank Sum Test  & \textbf{2} & 47 & 27 \\
        SF CUR - AIC & 3 & \textbf{37} & 49 \\
        SF CUR - BIC  & 7 & 99 & 27 \\
        LS-D CUR - AIC  & 6 & 89 & 29 \\
        LS-D CUR - BIC  & 8  & 109 & 28 \\
        DEIM CUR - AIC  & \textbf{2}  & 42 & 26 \\
        DEIM CUR - BIC  & \textbf{2}  & 42 & 26 \\
        QR CUR - AIC &  3 & 52 & 25 \\
        QR CUR - BIC &  4 & 67 & 20 \\
        \hline
    \end{tabular}
    \caption{Experiment 2 results.  The minimum numbers of mixed-SC-class neurons and observations are in bold.}
    \label{tab:SC_res}
\end{table}

The best performing feature selection methods are the DEIM CUR - AIC and DEIM CUR - BIC, each resulting in two mixed-SC-class neurons containing 42 observations.  The Wilcoxon rank-sum test also performs well, resulting in two mixed-SC-class neurons containing 47 observations.  While the SF CUR - AIC results in three mixed-SC-class neurons, it does achieve the minimum number of observations in mixed-SC-class neurons with 37; however, it selects 49 discriminating proteins which is at least 20 more than all other feature selection methods.  The QR CUR - AIC performs relatively well; however, not as well as the Wilcoxon rank-sum test since it results in three mixed-SC-class neurons containing 52 observations. The SF CUR - BIC, LS-D CUR - AIC, and LS-D CUR - BIC perform poorly compared to the other feature selection methods.  In addition, these three methods and the QR CUR - BIC perform worse than the baseline of no feature selection. 

The results of Experiments 1 and 2 demonstrate that CUR-based feature selection can be an effective feature selection method for this application, but results are data and CUR algorithm dependent.  Nonetheless, we demonstrated that DEIM-CUR is an excellent option for feature selection in Experiment 2.

\section{Conclusion}
\label{sec:cur_conc}

We have presented SF CUR, a novel CUR matrix approximation using convex optimization with supporting theory and numerical experiments. Specifically, the SF CUR algorithm uses the surrogate functional \cite{daubechies_etal_2004} to solve the convex optimization problems that arise in the method.  To the best of our knowledge, this is the only CUR method using convex optimization that solves for $\mathbf{C}$ and $\mathbf{R}$ separately and allows the user to choose the number of columns and rows for inclusion in $\mathbf{C}$ and $\mathbf{R}$, respectively.  In addition, we extended the theory of the surrogate functional technique to apply to SF CUR.  
We numerically demonstrated the use of SF CUR on sparse and dense data.  Specifically, we (1) calculated its relative error and computation time on a document-term dataset and a gene expression dataset, and (2) used it as a feature selection method on the gene expression dataset in order to classify patients as those with or without a tumor as in \cite{sorenson_embree_2016}.  We compared SF CUR performance on these numerical tasks to the SVD and other complementary CUR approximations.  We found that while the SVD provides the optimal approximation to each dataset, SF CUR performed the best in selecting probes to separate patient classes on the gene expression dataset, with LS-D CUR a close second in performance.

We also presented a novel application of CUR to determine discriminant proteins when clustering protein expression data in an SOM.  These experiments were based on those in \cite{higuera_et_al_2015}, with the exception that CUR was used as the feature selection method instead of the Wilcoxon rank-sum test.  We compared the performance of SF CUR with that of the Wilcoxon rank-sum test and other complementary CUR approximations.  We performed two experiments and found that performance varied between datasets and CUR algorithms.  While CUR-based feature selection performance was generally poor in Experiment 1, multiple CUR algorithms performed well in Experiment 2.  In fact, DEIM CUR - AIC and DEIM CUR - BIC were the best performing feature selection methods in Experiment 2.  In addition, to the best of our knowledge, this was the first use of CUR on protein expression data.  

Potential areas for future work include generalizing the SF CUR objective function as mentioned in Section \ref{sec:cur_gen}, implementation speed-ups for SF CUR on sparse data, and investigating why certain CUR algorithms perform better than others in particular applications or on particular datasets.

\section*{Conflict of Interest Statement}
The authors declare that the research was conducted in the absence of any commercial or financial relationships that could be construed as a potential conflict of interest.

\section*{Funding}
Radu Balan was partially supported by NSF under DMS-2108900, and by Simons Fellow award 818333.

\section*{Acknowledgments}
The authors acknowledge Research Computing at The University of Virginia for providing computational resources that have contributed to the results reported within this paper (URL: \url{https://rc.virginia.edu}). 

\section*{Data Availability Statement}
The original contributions presented in the study are included in the article/supplementary material, further inquiries can be directed to the corresponding author.

\section*{Generative AI Statement}
The authors declare that no Gen AI was used in the creation of this manuscript.

\bibliography{refs}

\begin{thebibliography}{26}
\expandafter\ifx\csname natexlab\endcsname\relax\def\natexlab#1{#1}\fi
\expandafter\ifx\csname urlstyle\endcsname\relax
  \expandafter\ifx\csname doi\endcsname\relax
  \def\doi#1{doi:\discretionary{}{}{}#1}\fi \else
  \expandafter\ifx\csname doi\endcsname\relax
  \def\doi{doi:\discretionary{}{}{}\begingroup \urlstyle{rm}\Url}\fi \fi
\expandafter\ifx\csname selectlanguage\endcsname\relax
  \def\selectlanguage#1{}\fi

\bibitem[{Mahoney and Drineas(2009)}]{mahoney_drineas_2009}
Mahoney MW, Drineas P.
\newblock {CUR} matrix decompositions for improved data analysis.
\newblock {\em Proceedings of the National Academy of Sciences\/} {\bf 106} (2009) 697--702.
\newblock \doi{10.1073/pnas.0803205106}.

\bibitem[{Sorenson and Embree(2016)}]{sorenson_embree_2016}
Sorenson DC, Embree M.
\newblock A {DEIM} induced {CUR} factorization.
\newblock {\em SIAM Journal on Scientific Computing\/} {\bf 38} (2016) A1454--A1482.
\newblock \doi{10.1137/140978430}.

\bibitem[{Liu and Shao(2010)}]{liu_shao_2010}
Liu Y, Shao J.
\newblock High dimensionality reduction using {CUR} matrix decomposition and auto-encoder for web image classification.
\newblock Qiu G, Lam KM, Kiya H, Xue XY, Kuo CCJ, Lew MS, editors, {\em Advances in Multimedia Information Processing - PCM 2010\/} (Berlin, Heidelberg: Springer Berlin Heidelberg) (2010), 1--12.

\bibitem[{Esmaeili et~al.(2021)Esmaeili, Joneidi, Salimitari, Khalid, and Rahnavard}]{esmaeili_etal_2021}
Esmaeili A, Joneidi M, Salimitari M, Khalid U, Rahnavard N.
\newblock Two-way spectrum pursuit for {CUR} decomposition and its application in joint column/row subset selection.
\newblock {\em 2021 IEEE 31st International Workshop on Machine Learning for Signal Processing (MLSP)\/} (2021), 1--6.
\newblock \doi{10.1109/MLSP52302.2021.9596233}.

\bibitem[{Li et~al.(2019)Li, Wang, Dong, Yan, Liu, and Zha}]{li_etal_2019}
Li C, Wang X, Dong W, Yan J, Liu Q, Zha H.
\newblock Joint active learning with feature selection via {CUR} matrix decomposition.
\newblock {\em IEEE Transactions on Pattern Analysis and Machine Intelligence\/} {\bf 41} (2019) 1382--1396.
\newblock \doi{10.1109/TPAMI.2018.2840980}.

\bibitem[{Hamm and Huang(2020)}]{hamm_huang1_2020}
Hamm K, Huang L.
\newblock Perspectives on {CUR} decompositions.
\newblock {\em Applied and Computational Harmonic Analysis\/} {\bf 48} (2020) 1088 -- 1099.
\newblock \doi{10.1016/j.acha.2019.08.006}.

\bibitem[{Goreinov et~al.(1997)Goreinov, Tyrtyshnikov, and Zamarashkin}]{goreinov_etal_1997}
Goreinov S, Tyrtyshnikov E, Zamarashkin N.
\newblock A theory of pseudoskeleton approximations.
\newblock {\em Linear Algebra and its Applications\/} {\bf 261} (1997) 1--21.
\newblock \doi{10.1016/S0024-3795(96)00301-1}.

\bibitem[{Drineas et~al.(2006)Drineas, Kannan, and Mahoney}]{drineas_etal_2006}
Drineas P, Kannan R, Mahoney MW.
\newblock Fast monte carlo algorithms for matrices {III}: Computing a compressed approximate matrix decomposition.
\newblock {\em SIAM Journal on Computing\/} {\bf 36} (2006) 184--206.
\newblock \doi{10.1137/S0097539704442702}.

\bibitem[{Drineas et~al.(2008)Drineas, Mahoney, and Muthukrishnan}]{drineas_etal_2008}
Drineas P, Mahoney MW, Muthukrishnan S.
\newblock Relative-error ${CUR}$ matrix decompositions.
\newblock {\em SIAM Journal on Matrix Analysis and Applications\/} {\bf 30} (2008) 844--881.
\newblock \doi{10.1137/07070471X}.

\bibitem[{Stewart(1999)}]{stewart_1999}
Stewart G.
\newblock Four algorithms for the the efficient computation of truncated pivoted {QR} approximations to a sparse matrix.
\newblock {\em Numerische Mathematik\/} {\bf 83} (1999) 313--323.
\newblock \doi{10.1007/s002110050451}.

\bibitem[{Mairal et~al.(2011)Mairal, Jenatton, Obozinski, and Bach}]{mairal_etal_2011}
Mairal J, Jenatton R, Obozinski G, Bach F.
\newblock Convex and network flow optimization for structured sparsity.
\newblock {\em J. Mach. Learn. Res.\/} {\bf 12} (2011) 2681–2720.

\bibitem[{Bien et~al.(2010)Bien, Xu, and Mahoney}]{bien_xu_mahoney_2010}
Bien J, Xu Y, Mahoney MW.
\newblock {CUR} from a sparse optimization viewpoint.
\newblock {\em Proceedings of the 23rd International Conference on Neural Information Processing Systems - Volume 1\/} (Red Hook, NY, USA: Curran Associates Inc.) (2010), NeurIPS'10, 217–225.

\bibitem[{Daubechies et~al.(2004)Daubechies, Defrise, and {De Mol}}]{daubechies_etal_2004}
Daubechies I, Defrise M, {De Mol} C.
\newblock An iterative thresholding algorithm for linear inverse problems with a sparsity constraint.
\newblock {\em Communications on Pure and Applied Mathematics\/} {\bf 57} (2004) 1413–1457.

\bibitem[{Higuera et~al.(2015{\natexlab{a}})Higuera, Gardiner, and Cios}]{higuera_et_al_2015}
Higuera C, Gardiner KJ, Cios KJ.
\newblock Self-organizing feature maps identify proteins critical to learning in a mouse model of down syndrome.
\newblock {\em PLOS ONE\/} {\bf 10} (2015{\natexlab{a}}) 1--28.
\newblock \doi{10.1371/journal.pone.0129126}.

\bibitem[{Dong and Martinsson(2023)}]{dong_martinsson_2023}
Dong Y, Martinsson PG.
\newblock Simpler is better: a comparative study of randomized pivoting algorithms for {CUR} and interpolative decompositions.
\newblock {\em Advances in Computational Mathematics\/} {\bf 49} (2023).
\newblock \doi{10.1007/s10444-023-10061-z}.

\bibitem[{Ida et~al.(2020)Ida, Kanai, Fujiwara, Iwata, Takeuchi, and Kashima}]{ida_etal_2020}
Ida Y, Kanai S, Fujiwara Y, Iwata T, Takeuchi K, Kashima H.
\newblock Fast deterministic {CUR} matrix decomposition with accuracy assurance.
\newblock Daumé~III H, Singh A, editors, {\em Proceedings of the 37th International Conference on Machine Learning\/} (PMLR) (2020), vol. 119, 4594--4603.

\bibitem[{Peng et~al.(2018)Peng, Luo, Li, Liu, and Zheng}]{peng_etal_2018}
Peng Z, Luo M, Li J, Liu H, Zheng Q.
\newblock {ANOMALOUS}: A joint modeling approach for anomaly detection on attributed networks.
\newblock {\em Proceedings of the Twenty-Seventh International Joint Conference on Artificial Intelligence, {IJCAI-18}\/} (International Joint Conferences on Artificial Intelligence Organization) (2018), 3513--3519.
\newblock \doi{10.24963/ijcai.2018/488}.

\bibitem[{Grant and Boyd(2014)}]{cvx_web}
Grant M, Boyd S.
\newblock {CVX}: Matlab software for disciplined convex programming, version 2.1 (2014).
\newblock \url{http://cvxr.com/cvx}.

\bibitem[{Grant and Boyd(2008)}]{cvx_book}
Grant M, Boyd S.
\newblock Graph implementations for nonsmooth convex programs.
\newblock Blondel V, Boyd S, Kimura H, editors, {\em Recent Advances in Learning and Control\/} (Springer-Verlag Limited), Lecture Notes in Control and Information Sciences (2008), 95--110.
\newblock \url{http://stanford.edu/~boyd/graph_dcp.html}.

\bibitem[{Parikh and Boyd(2014)}]{parikh_boyd_2014}
Parikh N, Boyd S.
\newblock Proximal algorithms.
\newblock {\em Found. Trends Optim.\/} {\bf 1} (2014) 127–239.
\newblock \doi{10.1561/2400000003}.

\bibitem[{dat(n.d.)}]{data_20news}
[Dataset] {20 Newsgroups, ``by date'' version} (n.d.).
\newblock \url{http://people.csail.mit.edu/jrennie/20Newsgroups/}.

\bibitem[{Pedregosa et~al.(2011)Pedregosa, Varoquaux, Gramfort, Michel, Thirion, Grisel et~al.}]{scikit_2011}
Pedregosa F, Varoquaux G, Gramfort A, Michel V, Thirion B, Grisel O, et~al.
\newblock Scikit-learn: Machine learning in {P}ython.
\newblock {\em Journal of Machine Learning Research\/} {\bf 12} (2011) 2825--2830.

\bibitem[{Landi et~al.(2008)Landi, Dracheva, Rotunno, Figueroa, Liu, Dasgupta et~al.}]{data_gene_exp}
[Dataset] Landi M, Dracheva T, Rotunno M, Figueroa J, Liu H, Dasgupta A, et~al.
\newblock Gene expression signature of cigarette smoking and its role in lung adenocarcinoma development and survival.
\newblock Gene Expression Omnibus (2008).
\newblock \url{https://www.ncbi.nlm.nih.gov/geo/query/acc.cgi?acc=GSE10072}.

\bibitem[{Higuera et~al.(2015{\natexlab{b}})Higuera, Gardiner, and Cios}]{data_pro_exp}
[Dataset] Higuera C, Gardiner K, Cios K.
\newblock {Mice Protein Expression}.
\newblock UCI Machine Learning Repository (2015{\natexlab{b}}).
\newblock {d}oi:10.24432/C50S3Z.

\bibitem[{Akaike(1973)}]{akaike_1973}
Akaike H.
\newblock Information theory and an extension of the maximum likelihood principle.
\newblock Petrov B, Cs\'{a}ki F, editors, {\em 2nd International Symposium on Information Theory\/} (Budapest, Hungary: Akad\'{e}mia Kiad\'{o}) (1973), 267--281.

\bibitem[{Schwarz(1978)}]{schwarz_1978}
Schwarz G.
\newblock Estimating the dimension of a model.
\newblock {\em The Annals of Statistics\/} {\bf 6} (1978) 461--464.
\newblock \doi{10.1214/aos/1176344136}.

\end{thebibliography}

\end{document}